\documentclass[a4paper,12pt,reqno]{amsart}

\usepackage{mathtools}
\usepackage{amsmath}
\usepackage{amssymb}
\usepackage{amsfonts}
\usepackage{graphicx}

\usepackage[colorlinks]{hyperref}
\renewcommand\eqref[1]{(\ref{#1})} 

\usepackage[margin=0.81in]{geometry}

\graphicspath{ {images/} }

\newtheoremstyle{theorem}
{10pt}          
{10pt}  
{\sl}  
{\parindent}     
{\bf}  
{. }    
{ }    
{}     
\theoremstyle{theorem}

\numberwithin{equation}{section}
\theoremstyle{plain}
\newtheorem{thm}{Theorem}[section]
\newtheorem{prop}[thm]{Proposition}
\newtheorem{cor}[thm]{Corollary}

\theoremstyle{definition}
\newtheorem{defn}[thm]{Definition}
\newtheorem{rem}[thm]{Remark}

\title[Generalized  fractional Dirac type operators]{Generalized  fractional Dirac type operators}

\author[J. E. Restrepo]{Joel E. Restrepo}
\address{
	Joel E. Restrepo:
	\endgraf
	Department of Mathematics: Analysis, Logic and Discrete Mathematics
	\endgraf
	Ghent University, Krijgslaan 281, Building S8, B 9000 Ghent
	\endgraf
	Belgium
	\endgraf
	{\it E-mail address} {\rm cocojoel89@yahoo.es; joel.restrepo@ugent.be}
}

\author[M. Ruzhansky]{Michael Ruzhansky}
\address{Michael Ruzhansky:
	\endgraf
	Department of Mathematics: Analysis, Logic and Discrete Mathematics
	\endgraf
	Ghent University, Krijgslaan 281, Building S8, B 9000 Ghent
	\endgraf
	Belgium
	\endgraf
	and
	\endgraf
	School of Mathematical Sciences
		\endgraf Queen Mary University of London 
			\endgraf
		United Kingdom
			\endgraf
	{\it E-mail address} {\rm Michael.Ruzhansky@ugent.be}
}

\author[D. Suragan]{Durvudkhan Suragan}
\address{
	Durvudkhan Suragan:
	\endgraf
	Department of Mathematics
	\endgraf
	Nazarbayev University
	\endgraf
	Kazakhstan
	\endgraf
	{\it E-mail address} {\rm durvudkhan.suragan@nu.edu.kz}}

\thanks{The paper has been accepted for publication in the journal ``Fractional Calculus and Applied Analysis".}

\numberwithin{equation}{section}

\usepackage{mathtools}

\usepackage{mathtools}
\usepackage{amsmath}
\usepackage{amssymb}
\usepackage{amsfonts}
\usepackage{graphicx}

\DeclareMathOperator*{\esssup}{ess\,sup}

\newcommand{\bx}{\boldsymbol{x}}
\newcommand{\bs}{\boldsymbol{s}}
\newcommand{\bk}{\boldsymbol{k}}
\newcommand{\bl}{\boldsymbol{\ell}}

\newcommand{\by}{\boldsymbol{y}}

\newcommand{\bq}{\boldsymbol{q}}

\newcommand{\bphi}{\boldsymbol{\phi}}

\begin{document}

\begin{abstract}
We introduce a class of fractional Dirac type operators with time variable coefficients by means of a Witt basis, the Djrbashian-Caputo fractional derivative and the fractional Laplacian, both operators defined with respect to some given functions. Direct and inverse fractional Cauchy type problems are studied for the introduced operators. We give explicit solutions of the considered fractional Cauchy type problems. We also use a recent method to recover a variable coefficient solution of some inverse fractional wave and heat type equations. Illustrative examples are provided.
\end{abstract} 

\maketitle
\tableofcontents

\section{Introduction}

The Dirac type operators have been well-known by their great impact and applications in Clifford analysis and PDE's, we refer to the books \cite{[4],[6],cerebook,otrobookd,[17]}, and also, just to mention a few, the papers \cite{[3a],calleje,2009,ferreira2017mmas,ferreira2017,[16a]}. For some works related to more general presentations and applications of Dirac type operators, see e.g. \cite{gene1,ferreira2018,gene2,generalizedintro}.

In this paper we follow, generalize and extend some of the ideas given in \cite{ferreira2017mmas,ferreira2017} where fundamental solutions (in the multidimensional case) of time-fractional telegraph, diffusion-wave and parabolic Dirac operators were obtained. In fact, the authors used some operational techniques via Fourier and Mellin transforms to get explicit representations of the solutions. In addition, they used a Witt basis to define a fractional Dirac type operator which allows one to factorize their equations and get explicit solutions in the setting of Clifford analysis. This idea was used in the classical case (standard derivatives) in \cite{calleje}. 

We mainly introduce a class of fractional Dirac type operators by means of a Witt basis that factorize a general fractional Laplace-type operator and involves the Djrbashian-Caputo fractional derivatives with respect to another function, the recent introduced  multidimensional fractional Laplacian with respect to a function \cite[Definition 4.9]{fractional-laplacian} and time variable  coefficients. Notice that a similar Dirac type operator was first introduced in \cite{BRS}. Nevertheless, in this article, we define the operator in a more general form that gives the possibility to obtain several results of a huge class of (fractional) partial differential equations which in particular recover the ones from \cite{BRS}. It is important to highlight that we cover heat and wave type equations as a special case. These type of Dirac operators can be very useful to analyze the solvability of the in-stationary Navier--Stokes equations as in  \cite{calleje}, as well as Maxwell equations, Lame equations, among others \cite{1spro,bookclifford1}. In some theoretical frames, our results and the generalized fractional Dirac type operators will allow (in the future) to explore different questions between fractional calculus and some topics like Clifford analysis, quantum mechanics, physics, etc \cite{[10a],[14a],[17a],[19a]}.

By using the new Dirac type operator we study some direct and inverse fractional Cauchy type problems in the setting of Clifford analysis. We consider Cauchy type problems of similar type as those in \cite{BRS,fractional-laplacian,generalizedCauchyRS}. Fractional direct and inverse Cauchy type problems have been studied by many authors since their applications and the intrinsic development of the fractional calculus theory. We refer, for instance, to the sources \cite{inve2,si5,inve5,inve4,inve6,RTT,si3,inve7,si1,inve1,si4} and references therein. The following books \cite{diethelm,kilbas,bookana,samko,fee} are of relevance as well. 
       
Let us now mention some features of the new Dirac type operator (see formula \eqref{classdiracnew}) which have not been considered anywhere else: 
\begin{enumerate}
    \item The fractional operators are considered with respect to another function and complex orders. In this context, the fractional integral was introduced by Holmgren (1866) with somehow influence of Liouville's work (1835). Properties and other features of this integral were considered by Sewell (1937), Shelkovnikov (1951) and Chrysovergis (1971).  While, the differential operator was introduced by Erd\'elyi (1964,1970) as well as Talenti (1965). We also point out that fractional integral of this type in the complex plane was considered by Osler (1970,1972). All details of these historical notes and other information can be found in \cite[Notes to 18.2, pp. 431-432]{samko}. These operators give the possibility to obtain results e.g. of Hadamard fractional calculus, Erd\'elyi-Kober fractional calculus and the classical one. More details can be found in Subsection \ref{intro-fc}. Also, some  other classes of different types can be considered by varying the given function. On the other hand, the complex fractional orders give a different nature (not well-known yet) to the considered operator. These type of orders were studied by Kober in \cite{kober} and Love in \cite{[11],[12]}. Fractional differential-integral operators played an important role for hypergeometric integral equations whose solutions involved fractional integrals and derivatives of complex orders, see e.g. the following classical works \cite{[11],[12],ross}, some more recent ones such as \cite{r3,r1,r4,r2} and references therein.  
    \item The multidimensional fractional Laplacian with respect to a given function (see Subsection \ref{Secsub:Fwrtfn}). This operator was recently introduced in \cite{fractional-laplacian} and it is still under review for its potential for future applications. At least here we give one. In particular it becomes the classical fractional Laplacian.  
    \item Time variable coefficients. Usually operators are considered with constant coefficients, which gives more freedom to work and analyze different objects. Nevertheless, we handle and allow here to consider continuous time variable  coefficients.   
\end{enumerate}
                     
The paper is organized as follows: In Section \ref{preli}, we recall some facts and definitions on the fractional Laplacian, Fourier transform, fractional Cauchy type equations and Clifford analysis. Section \ref{mainresults} is devoted to the main results of the paper. Indeed, by using a class of generalized time-fractional Dirac type operators, we study fractional Cauchy type problems and give their explicit solutions. Some examples are also given. In Section \ref{special} we discuss some special cases of the introduced Dirac type operators. While, in Section \ref{directpro}, we study some inverse fractional wave and heat type equations. We complement it with some examples.

\section{Preliminaries}\label{preli} 

In this section we collect several definitions and results about the fractional Laplacian, Fourier transform, fractional Cauchy type equations and Clifford analysis, which will be needed throughout the paper.   

\subsection{The fractional Laplacian}

We first fix some notations and recall some known results regarding the fractional Laplacian. These results can be found for example in \cite{Gris,laplacian-book1,laplacian-book2} and the references therein.

We recall some standard definitions of function spaces on $\mathbb{R}^n$. For any $p\in[1,+\infty)$, the $L^p$ space is defined as
\[
L^{p}(\mathbb{R}^n)\coloneqq \left\{f:\mathbb{R}^n\to\mathbb{R}\;\mbox{ measurable such that}\; \int_{\mathbb{R}^n}|f(\bx)|^{p}\;\mathrm{d}\bx<+\infty\right\},
\]
which is a Banach space under the standard $L^p$ norm. For $p=+\infty$, as usual we have:
\[
L^{\infty}(\mathbb{R}^n)\coloneqq \left\{f:\mathbb{R}^n\to\mathbb{R}\;\mbox{ measurable such that}\; \displaystyle\esssup_{\bx\in\mathbb{R}^n}|f(\bx)|<+\infty\right\},
\]
which is also a Banach space.

The weighted $L^{p}$ spaces can be defined, in particular, for any non-negative measurable function $g:\mathbb{R}^n\to\mathbb{R}$ as follows:
\[
L^{p}(\mathbb{R}^n,g\,\mathrm{d}\bx)\coloneqq \left\{f:\mathbb{R}^n\to\mathbb{R}\;\mbox{ measurable and}\; \int_{\mathbb{R}^n}|f(\bx)|^{p}g(\bx)\;\mathrm{d}\bx<+\infty\right\},
\]
endowed with the norm given by the $(1/p)$th power of the above integral. Here $g$ can be taken as the Radon--Nikod\'ym derivative of one $\sigma$-finite measure with respect to another, transforming between two Banach spaces with different weights.

The Schwartz space $\mathcal{S}(\mathbb{R}^n)$ is the space of smooth and rapidly decreasing functions:
\[
\mathcal{S}(\mathbb{R}^n)\coloneqq \left\{f\in C^{\infty}(\mathbb{R}^n)\;\mbox{ and}\; \sup_{\bx\in\mathbb{R}^n}\left|\bx^{\bk}\partial^{\bl}f(\bx)\right|<+\infty\;\mbox{ for all}\;\bk,\bl\in\mathbb{N}^n\right\},
\]
and the topology on this space is defined by the following semi-norms:
\[
p_{N}(f)=\sup _{\bx \in \mathbb{R}^{n}}(1+|\bx|)^{N} \sum_{|\bk| \leqslant N}\left|\partial^{\bk}f(\bx)\right|,\quad f\in\mathcal{S}(\mathbb{R}^n),
\]
where $N$ can take any value in $\mathbb{Z}^+_0$. The $L^p$ and Schwartz spaces are useful and basic framework to work with the fractional Laplacian and similar operators.

To give a proper definition of the fractional Laplacian via the principal value integral, we use the following function space for $0<s<1$ and $n\in\mathbb{N}$:
\[
L_s^{1}(\mathbb{R}^n)\coloneqq \left\{f:\mathbb{R}^n\to\mathbb{R}\;\mbox{ measurable such that}\; \int_{\mathbb{R}^n}\frac{|f(\bx)|}{(1+|\bx|)^{n+2s}}\;\mathrm{d}\bx<+\infty\right\}.
\]
For $f\in L_s^{1}(\mathbb{R}^n)$ and $\varepsilon>0$, we set
\begin{align*}
(-\Delta)_\varepsilon^sf(x)\coloneqq C_{n,s}\int_{\{\by\in\mathbb{R}^n:\;|\bx-\by|>\varepsilon\}}\frac{f(\bx)- f(\by)}{|\bx-\by|^{n+2s}}\;\mathrm{d}\by,\qquad\bx\in\mathbb{R}^n,
\end{align*}
where $C_{n,s}$ is a normalization constant, usually given by $C_{n,s}\coloneqq \frac{s2^{2s}\Gamma\left(s+\frac{n}{2}\right)}{\pi^{n/2}\Gamma(1-s)}.$ Thus, we can define the fractional Laplacian $(-\Delta)^s$ by the following integral:
\[
(-\Delta)^sf(\bx)\coloneqq C_{n,s}\,\mbox{P.V.}\int_{\mathbb{R}^n}\frac{f(\bx)-f(\by)}{|\bx-\by|^{n+2s}}\;\mathrm{d}\by = \lim_{\varepsilon\downarrow 0}(-\Delta)_\varepsilon^sf(\bx),
\]
whenever the limit exists for a.e. $\bx\in\mathbb{R}^n$ (\cite{NPV}).

\medskip In the next two subsections we recall several notions and facts on the Fourier transform and fractional Laplacian with respect to another function recently introduced in \cite{fractional-laplacian}.

\subsection{The Fourier transform and fractional Laplacian with respect to a given function} 

\subsubsection{The $n$-dimensional Fourier transform with respect to $\bphi$} \label{Secsub:Fwrtfn}

\begin{defn}\label{Def:Fwrtfn}
Let $n\in\mathbb{N}$, and let $\bphi:\mathbb{R}^n\to\mathbb{R}^n$ be a bijection all of whose $1$st-order partial derivatives exist almost everywhere. The $n$-dimensional Fourier transform with respect to $\bphi$, of a function $f:\mathbb{R}^n\to\mathbb{C}$ such that this integral exists, is defined by
\[
\big[\mathcal{F}_{\bphi}f\big](\bk)=\frac{1}{(2\pi)^{n/2}}\int_{\mathbb{R}^n}e^{-i\bk\cdot\bphi(\bx)}f(\bx)\big|J\bphi(\bx)\big|\,\mathrm{d}\bx,\quad \bk\in\mathbb{R}^n,
\]
where $J\bphi$ is the Jacobian determinant of the function $\bphi$. In particular, for $\bphi(\bx)=\bx$, we recover the classical Fourier transform in $\mathbb{R}^n$ which we denote by $\mathcal{F}.$ 
\end{defn}

Now we recall some basic properties of the above operator. Below we frequently use the operator  $Q_{\phi}$ of the right composition with $\phi$:
\[
\label{Qphi}
Q_{\phi}(f)=f\circ\phi,\qquad\text{ i.e. }\qquad(Q_{\phi}f)(x)=f(\phi(x)).
\]

\begin{prop}\cite[Proposition 3.19]{fractional-laplacian} \label{Prop:Fwrtfn:oprel}
The $n$-dimensional Fourier transform with respect to another function is related to the classical Fourier transform by the following equality:
\[
\mathcal{F}_{\bphi}=\mathcal{F}\circ Q_{\bphi}^{-1},
\]
where $Q_{\bphi}^{-1}$ is the operator of right composition with $\phi^{-1}$, i.e. $(Q_{\phi}^{-1}f)(\bx)=f(\phi^{-1}(\bx)).$
\end{prop}

\begin{prop}\cite[Proposition 3.20]{fractional-laplacian}
The $n$-dimensional inverse Fourier transform with respect to $\bphi$ is given by
\[
\big[\mathcal{F}_{\bphi}^{-1}g\big](\bx)=\frac{1}{(2\pi)^{n/2}}\int_{\mathbb{R}^n}e^{i\bk\cdot\bphi(\bx)}g(\bk)\,\mathrm{d}\bk,
\]
which means that $g=\mathcal{F}_{\bphi}f\iff f=\mathcal{F}_{\bphi}^{-1}g$.
\end{prop}

\subsubsection{\bf The $n$-dimensional Fourier convolution operator with respect to $\bphi$}

\begin{defn}
Let $n\in\mathbb{N}$, and let $\bphi:\mathbb{R}^n\to\mathbb{R}^n$ be a bijection all of whose $1$st-order partial derivatives exist almost everywhere. The Fourier-type $\bphi$-convolution of two functions $f$ and $g$ in the space $L^1(\mathbb{R}^n,\mathrm{d}\bphi)$ is defined as follows:
\[
h(\bx)=\big(f*_{\bphi}g\big)(\bx)=\int_{\mathbb{R}^n}f\big(\bphi^{-1}\big(\bphi(\bx)-\bphi(\by)\big)\big)g(\by)\big|J\bphi(\by)\big|\,\mathrm{d}\by.
\]
\end{defn}

\subsubsection{\bf The $n$-dimensional fractional Laplacian with respect to $\bphi$} \label{Secsub:FLwrtfn}

\begin{defn}\label{Def:FLnwrtf}
Let $n\in\mathbb{N}$, and let $\bphi:\mathbb{R}^n\to\mathbb{R}^n$ be a bijection all of whose $1$st-order partial derivatives exist almost everywhere. The $n$-dimensional fractional Laplacian with respect to $\bphi$, of a function $f$ defined on $\mathbb{R}^n$, is
\[
\left(-\Delta\right)^{\alpha/2}_{\bphi(\bx)}f(\bx)=C_{n,\alpha/2}\,\mbox{P.V.}\int_{\mathbb{R}^n}\frac{f(\bx)-f(\by)}{\big|\bphi(\bx)-\bphi(\by)\big|^{n+\alpha}}\big|J\bphi(\by)\big|\,\mathrm{d}\by,
\]
where $0<\alpha<2$, $\bx\in\mathbb{R}^n$ and $C_{n,\alpha/2}=\frac{\alpha2^{\alpha-1}\Gamma\left(\frac{\alpha+n}{2}\right)}{\pi^{n/2}\Gamma\left(1-\frac{\alpha}{2}\right)}.$
\end{defn}

\begin{rem}
In the $1$-dimensional case, the Definition \ref{Def:FLnwrtf} of the fractional Laplacian with respect to $\bphi:\mathbb{R}\to\mathbb{R}$ coincides up to some multiplicative factor (which vary under some parameters) with the Riesz-Feller fractional derivatives. We recommend to have a look on \cite[Section 2]{inve6} and also \cite[Section 3]{fractional-laplacian} to see clearly the connection. Nevertheless, in higher dimensions, this definition seems new.        
\end{rem}

Let us recall two important results about the above operator.

\begin{thm}\cite[Theorem 4.11]{fractional-laplacian} \label{Thm:MFLn:conjug}
The $n$-dimensional fractional Laplacian with respect to a function is related to the classical fractional Laplacian by the following conjugation relation:
\begin{align*}
\left(-\Delta\right)^{\alpha/2}_{\bphi(\bx)}=Q_{\bphi}\circ\left(-\Delta\right)^{\alpha/2}\circ Q_{\bphi}^{-1}, 
\end{align*}
where $0<\alpha<2$ and $\bphi:\mathbb{R}^n\to\mathbb{R}^n$ is a bijection all of whose $1$st-order partial derivatives exist almost everywhere.
\end{thm}

\begin{thm}\cite[Theorem 4.14]{fractional-laplacian}\label{flwrf}
Let $\bphi:\mathbb{R}^n\to\mathbb{R}^n$ be a bijection all of whose $1$st-order partial derivatives exist almost everywhere. The fractional Laplacian with respect to $\bphi$ and the Fourier transform with respect to $\bphi$ are related by the following identity:
\[
\mathcal{F}_{\bphi}\big[\left(-\Delta\right)^{\alpha/2}_{\bphi(\bx)}f(\bx)\big]=|\bk|^{\alpha}\mathcal{F}_{\bphi}f(\bk),
\]
for any $f\in L^p(\mathbb{R}^n,\mathrm{d}\bphi)$, $0<\alpha<2$, $1\leqslant p\leqslant 2$.
\end{thm}

\subsection{Fractional integro-differential operators} \label{intro-fc}
Now we recall some definitions and properties of the fractional integro \- differential operators with respect to another function, see e.g. \cite[Chapter 4]{samko}, also \cite{kilbas}.  

Below we fix a finite interval $[a,T]\subseteq\mathbb{R}$ and work over the following function spaces:
\begin{align*}
AC[a,T]&=\left\{f:[a,T]\to\mathbb{R}\;:\;f\text{ absolutely continuous on }[a,T]\right\}; \\
AC^n[a,T]&=\left\{f:[a,T]\to\mathbb{R}\;:\;f^{(n-1)}\text{ exists and in }AC[a,T]\right\},\qquad n\in\mathbb{N}.
\end{align*}

\begin{defn} 
Let $\alpha\in\mathbb{C}$, $\rm{Re}(\alpha)>0$, $-\infty\leqslant a<b\leqslant\infty$, $f\in L^1[a,b]$, and let $\psi\in C^1[a,b]$ be such that $\psi^{\prime}(t)>0$ for all $t\in[a,b]$. The left-sided Riemann-Liouville fractional integral of $f$ with respect to another function $\psi$ is defined by \cite[formula (2.5.1)]{kilbas}:
\[
I_{a+}^{\alpha,\psi}f(t)=\frac1{\Gamma(\alpha)}\int_a^t \psi^{\prime}(s)(\psi(t)-\psi(s))^{\alpha-1}f(s){\rm d}s.
\]
\end{defn}

\begin{defn}
Let $\alpha\in\mathbb{C}$, $\rm{Re}(\alpha)\geqslant0$, $-\infty\leqslant a<b\leqslant\infty$, and let $\psi\in C^1[a,b]$ be such that $\psi^{\prime}(t)>0$ for all $t\in[a,b]$. The left-sided Riemann-Liouville fractional derivative of a function $f$ with respect to another function $\psi$ is defined by  \cite[Formula 2.5.17]{kilbas}:
\[
D_{a+}^{\alpha,\psi}f(t)=\left(\frac{1}{\psi^{\prime}(t)}\frac{\rm d}{{\rm d}t}\right)^n \big(I_{a+}^{n-\alpha,\psi}f\big)(t),\quad f\in AC^n[a,b],
\]
where $n=\lfloor\rm{Re}(\alpha)\rfloor +1$ and  $\lfloor \cdot\rfloor$ is the floor function ($n-1\leqslant \rm{Re}(\alpha)<{\it n}$).
\end{defn}
Everywhere below we always assume that $\psi\in C^1[a,b]$ is such that $\psi^{\prime}(t)>0$ for all $t\in[a,b]$ when we use the operators $I_{a+}^{\alpha,\psi}$ or $D_{a+}^{\alpha,\psi}$.  

Let us recall a result which will be useful in some examples in the next sections. Taking into account \cite[Theorem 2.4]{samko} it can be proved similarly that the following statement holds. \begin{thm}\label{dirl}
If $\alpha\in\mathbb{C}$ $(\rm{Re}(\alpha)>0)$ and $f\in L^1(a,b)$, then
\[D_{a+}^{\alpha,\psi}I_{a+}^{\alpha,\psi}f(t)=f(t)\]
holds almost everywhere on $[a,b]$.
\end{thm}

In this paper we will use the following modified fractional derivative with respect to another function:
\begin{equation}\label{alternative}
^{C}D_{0+}^{\alpha,\psi}f(t)=D_{0+}^{\alpha,\psi}\biggl[f(t)-\sum_{j=0}^{n-1}\frac{\big(\psi(t)-\psi(0)\big)^j}{j!}\cdot\frac{\mathrm{d}^j}{\mathrm{d}\psi(t)^j}f(t)\big|_{_{_{t=0+}}}\big.\biggr],
\end{equation}
where $\alpha\in\mathbb{C}$, $\rm{Re}(\alpha)\geqslant0$, $n=\lfloor\rm{Re}(\alpha)\rfloor +1$ for $\alpha\notin\mathbb{N}$ and $n=\alpha$ for $\alpha\in\mathbb{N}$. 

\begin{rem}
In this paper we consider the initial point of the fractional operators to be zero. Nevertheless, it can be taken any positive real number.
\end{rem}

Note that for $\psi(t)=t$, $^{C}D_{0+}^{\alpha,\psi}f(t)$ becomes the modified fractional derivative used in \cite[formula (1.3)]{kilbas2004}. We also have: If $\alpha>0$ and $f\in AC^n[a,b]$, then $^{C}D_{0+}^{\alpha,\psi}$ of \eqref{alternative} becomes the so-called Djrbashian-Caputo fractional derivative: 
\begin{equation}\label{Capder}
^{C}{\mathcal D}^{\alpha,\psi}_{0+} {f(t)}=I_{0+}^{n-\alpha,\psi}\left(\frac1{\psi^{\prime}(t)}\frac{\rm d}{{\rm d}t}\right)^n f(t),
\end{equation}
where $n=\lfloor \alpha\rfloor+1$ for $\alpha\notin\mathbb{N}$ and $n=\alpha$ for $\alpha\in\mathbb{N}$. 

We must mention that the existence of the fractional derivative \eqref{Capder} is guaranteed by $f^{(n)}\in L^1[a,b]$. And, the stronger condition $f\in C^n[a,b]$ gives the continuity of the derivative. Furthermore, for $\alpha=n$, we have
\[
^{C}{\mathcal D}_{0+}^{\alpha,\psi}f(t)=\left(\frac1{\psi^{\prime}(t)}\frac{\rm d}{{\rm d}t}\right)^n f(t).
\]
For $\alpha=n$ and $\psi(t)=t$, it follows that $\prescript{C}{}{\mathcal D}_{0+}^{\alpha,\psi}f(t)=D^n f(t)=f^{(n)}(t).$

We use the Djrbashian-Caputo type derivatives instead of Riemann--Liouville type because of the variant behavior in initial conditions which are important in physical interpretations \cite{diethelm,physical}.

It is important to highlight here that  definitions \eqref{alternative} and \eqref{Capder} coincide for any function $f\in AC^n[a,T]$, see e.g. \cite[Theorem 3.1]{diethelm} or \cite[Theorem 2.2]{samko}. The crucial and main difference between both definitions is the possibility to define \eqref{alternative} in a larger function space than $AC^n[a,T]$. In fact, it is possible to define the Riemann--Liouville derivative on such larger function spaces. For a basic illustration of it, see \cite{example}, where some examples are shown on functions which have no first order derivative but have Riemann--Liouville fractional derivatives of all orders less than one.

Below we mention three special cases of the fractional calculus with respect to another function. More examples and properties can be found in \cite{almeida}. 

\begin{itemize}
\item For $\psi(x)=x$, fractional calculus with respect to this function is the same as the original fractional calculus.
\item For $\psi(x)=\log(x)$, fractional calculus with respect to this function is known as Hadamard fractional calculus. The operators of Hadamard fractional calculus are:
\begin{align*}
\prescript{H}{}I^{\alpha,\,\log(x)}_{0+}f(x)&=\frac{1}{\Gamma(\alpha)}\int_0^x\left(\log\frac{x}{y}\right)^{\alpha-1}\frac{f(y)}{y}\,\mathrm{d}y, \\
\prescript{HR}{}D^{\alpha,\,\log(x)}_{0+}f(x)&=\left(x\cdot\frac{\mathrm{d}}{\mathrm{d}x}\right)^n\prescript{H}{}I^{n-\alpha,\,\log(x)}_{0+}f(x), \\
\prescript{HC}{}D^{\alpha,\,\log(x)}_{0+}f(x)&=\prescript{H}{}I^{n-\alpha,\,\log(x)}_{0+}\left(x\cdot\frac{\mathrm{d}}{\mathrm{d}x}\right)^nf(x).
\end{align*}
\item For $\psi(x)=x^{\rho}$, fractional calculus with respect to this function was proposed by Erd\'elyi \cite{erdelyi} in 1964. The operators here are:
\begin{align*}
\prescript{RL}{}I^{\alpha,\,x^{\rho}}_{0+}f(x)&=\frac{\rho}{\Gamma(\alpha)}\int_0^x\big(x^{\rho}-y^{\rho}\big)^{\alpha-1}y^{\rho-1}f(y)\,\mathrm{d}y, \\
\prescript{RL}{}D^{\alpha,\,x^{\rho}}_{0+}f(x)&=\left(\frac{1}{\rho x^{\rho-1}}\cdot\frac{\mathrm{d}}{\mathrm{d}x}\right)^n\prescript{RL}{}I^{n-\alpha,\,x^{\rho}}_{0+}f(x), \\
\prescript{C}{}D^{\alpha,\,x^{\rho}}_{0+}f(x)&=\prescript{RL}{}I^{n-\alpha,\,x^{\rho}}_{0+}\left(\frac{1}{\rho x^{\rho-1}}\cdot\frac{\mathrm{d}}{\mathrm{d}x}\right)^nf(x).
\end{align*}
\end{itemize}

When additional weight parameter is added above, these are the so-called Erd\'{e}lyi-Kober (generalized) operators of fractional calculus, introduced by Sneddon in 1966 and studied further by Kiryakova \cite[Chapter 2]{kir}.

\subsection{Fractional Cauchy type problem} Here we recall some useful results from \cite{fractional-laplacian} that will help us to prove our main results in the next sections. 

We use the following formula for the Djrbashian-Caputo fractional derivative of a function with respect to another function in the time domain:
\[
\prescript{C}{}\partial^{\alpha,\psi}_{t}w(\bx,t)=D^{\alpha,\psi}_{0+}\left[w(\bx,t)-\sum_{j=0}^{n-1}\frac{\big(\psi(t)-\psi(0)\big)^j}{j!}\cdot\frac{\mathrm{d}^j}{\mathrm{d}\psi(t)^j}w(\bx,t)\big|_{_{_{t=0+}}}\big.\right],
\]
where $\alpha\in\mathbb{C}\setminus\mathbb{Z}$, $\rm{Re}(\alpha)\geqslant0$, $\bx\in\mathbb{R}^n$, $t\in(0,T]$, $n=\lfloor\rm{Re}(\alpha)\rfloor+1$ and $\psi\in C^1[0,T]$ is such that $\psi^{\prime}(t)>0$ for all $t\in[0,T].$ 

\subsubsection{\bf The variable-coefficient case} 
We consider the following space-time fractional Cauchy problem which involves both fractional time derivatives with respect to $\psi$ and the fractional Laplacian with respect to $\bphi$:
\begin{equation}\label{eq1Cauchynew}
\begin{split}
\prescript{C}{}\partial^{\beta_0,\psi}_{t}w(\bx,t)+\sum_{i=1}^{m-1}\rho_i(t) \prescript{C}{}\partial^{\beta_i,\psi}_{t}w(\bx,t)+\rho_{m}(t)\left(-\Delta\right)^{\alpha/2}_{\bphi(x)} w(\bx,t)&=h(\bx,t), \\
w(\bx,0)&=w_0(\bx), \\
\frac{\partial w}{\partial\psi(t)}(\bx,t)|_{_{_{t=0}}}&=w_1(\bx),\\
&\hspace{0.2cm}\vdots \\
\frac{\partial^{n_0-1}w}{\partial\psi(t)^{n_0-1}}(\bx,t)|_{_{_{t=0}}}&=w_{n_0-1}(\bx),
\end{split}
\end{equation}
where $t\in (0,T]$, $\bx\in\mathbb{R}^n$, $h(\cdot,t),\rho_i(t)\in C[0,T]$ $(1\leqslant i\leqslant m)$, $0<\alpha<2$, $\beta_i\in\mathbb{C}$ with $\rm{Re}(\beta_{\it i})>0$ for $0\leqslant i\leqslant m-1$, $\rm{Re}(\beta_0)>\rm{Re}(\beta_1)>\cdots>\rm{Re}(\beta_{\it m-1})>0$, $\Re (\beta_0)\not\in\mathbb{Z}$ and $n_i=\lfloor \rm{Re}\beta_{\it i}\rfloor+1$. We also suppose that $h(\bx,\cdot)\in L^1(\mathbb{R}^n,\mathrm{d}\bphi)$, $\mathcal{F}_{\bphi}h(\bx,\cdot)\in L^1(\mathbb{R}^n)$, and $w_{k}\in L^1(\mathbb{R}^n,\mathrm{d}\bphi)$, $\mathcal{F}_{\bphi}w_{k}\in L^1(\mathbb{R}^n)$, for $0\leqslant k\leqslant n_0-1$.

Moreover, we also need to assume that (for $\bk\in\mathbb{R}^n$)
\[
\sum_{k=0}^{+\infty}(-1)^{k}I^{\beta_0,\psi}_{0+}\left(\sum_{i=1}^{m-1}\rho_i(t)I^{\beta_0-\beta_i,\psi}_{0+}+|\bk|^{\alpha}\rho_m(t)I^{\beta_0,\psi}_{0+} \right)^{k}\mathcal{F}_{\bphi}h(\bk,t)\in L^1(\mathbb{R}^n).
\]

A solution of the space-time fractional Cauchy problem \eqref{eq1Cauchynew} is given in the space
\[
C^{n_0-1,\beta_0}[0,T]:=\left\{g(t)\in C^{n_0-1}[0,T]\;:\; \prescript{C}{}D^{\beta_0,\psi}_{0+}g(t)\in C[0,T]\right\},
\]
endowed with the norm
\[
\|g\|_{C^{n_0-1,\beta_0}[0,T]}=\sum_{k=0}^{n_0-1}\left\|\frac{\mathrm{d}^kg}{\mathrm{d}\psi(t)^k}\right\|_{C[0,T]}+\big\|\prescript{C}{}D^{\beta_0,\psi}_{0+}g\big\|_{C[0,T]}.
\] 
Now, for each $j=0,\ldots,n_0-1$, we define the function
\begin{equation}\label{PSI}
\Psi_j(t)=\frac{(\psi(t)-\psi(0))^{j}}{\Gamma(j+1)},\quad j\in{\mathbb N}\cup\{0\},
\end{equation}
and the set
\[
\mathbb{K}_j:=\{i:0\leqslant\rm{Re}(\beta_{\it i})\leqslant{\it j}, {\it i}=1,\dots,{\it m}\},\quad {\it j}=0,1,\dots,{\it n}_0-1,
\]
with $\kappa_j=\min\{\mathbb{K}_j\}$ if $\mathbb{K}_j\neq\emptyset$. Note that the inclusion $s\in\mathbb{K}_j$ implies $\rm{Re}(\beta_{\it s})\leqslant\it{j}$, while $\mathbb{K}_{j_1}\subset\mathbb{K}_{j_2}$ for $j_1<j_2$. Also, if $\beta_m=0$, then $\mathbb{K}_j\neq\emptyset$ for all $j=0,1,\ldots,n_0-1$. Following \cite{fractional-laplacian,RRS}, we define
\begin{align*}
\mathcal{K}_{j}&\left(t,|\bk|^{\alpha},\rho_1,\ldots,\rho_{m}\right):= \nonumber\\
&\sum_{k=0}^{+\infty} (-1)^{k+1}I^{\beta_0,\psi}_{0+}\left(\sum_{i=1}^{m}d_i(t)I^{\beta_0-\beta_i,\psi}_{0+}\right)^k \sum_{i=1}^{m}d_i(t)D^{\beta_i,\psi}_{0+}\Psi_j(t),
\end{align*}
and
\begin{align}
\mathcal{K}_{j}^{\kappa_j}&\left(t,|\bk|^{\alpha},\rho_1,\ldots,\rho_{m}\right):= \nonumber \\
&\sum_{k=0}^{+\infty} (-1)^{k+1}I^{\beta_0,\psi}_{0+}\left(\sum_{i=1}^{m}d_i(t)I^{\beta_0-\beta_i,\psi}_{0+}\right)^k \sum_{i=\kappa_j}^{m}d_i(t)D^{\beta_i,\psi}_{0+}\Psi_j(t), \label{kjkappa}
\end{align}
where $d_{m}(t)=|\bk|^{\alpha}\rho_{m}(t)$ and $d_i(t)=\rho_{i}(t)$ for $1\leqslant i\leqslant m-1$. We also define
\[
\mathcal{G}\big((\mathcal{F}_{\bphi}h)(\bk,t)\big):=\sum_{k=1}^{+\infty}(-1)^{k}I_{0+}^{\beta_0,\psi}\left(\sum_{i=1}^{m}d_i(t)I_{0+}^{\beta_0-\beta_i,\psi}\right)^{k}(\mathcal{F}_{\bphi}h)(\bk,t).
\] 
We always require that the kernels $\mathcal{K}_j^{\kappa_j}$ and $\mathcal{K}_j$ $(0\leqslant j\leqslant n_0-1)$ and $\mathcal{G}\big((\mathcal{F}_{\bphi}h)(\bk,t)\big)$ are in $L^1(\mathbb{R}^n)$, which is a natural condition even imposed in the classical case where there are not fractional operators involved.

\medskip Now we recall a result on the solution of the equation \eqref{eq1Cauchynew}. 

\begin{thm}\label{cor3.2.1CP} 
Let $h(\cdot,t),\rho_i(t)\in C[0,T]$ for $1\leqslant i\leqslant m$. Let  $h(\bx,\cdot)\in L^1(\mathbb{R}^n,\mathrm{d}\bphi)$, $\mathcal{F}_{\bphi}h(\bx,\cdot)\in L^1(\mathbb{R}^n)$, $w_k\in L^1(\mathbb{R}^n,\mathrm{d}\bphi)$ and $\mathcal{F}_{\bphi}w_k\in L^1(\mathbb{R}^n)$ for $0\leqslant k\leqslant n_0-1$. 
\begin{enumerate}
    \item If $n_0>n_1$, $\beta_{m}=0$, then the initial value problem \eqref{eq1Cauchynew} has a unique solution $w(\cdot,t)\in C^{n_0-1,\beta_0}[0,T]$ given by
\begin{multline*}
w(\bx,t)=\sum_{j=0}^{n_0-1}w_j(\bx)\Psi_j(t)+\sum_{j=0}^{n_0-1}\Big(w_j*_{\bphi}\mathcal{F}_{\bphi}^{-1}\big(\mathcal{H}_j(t,|\bk|^{\alpha},\rho_1,\ldots,\rho_{m})\big)\Big)(\bx) \\ -I_{0+}^{\beta_0,\psi}h(\bx,t)+\big(\mathcal{F}^{-1}_{\bphi}\mathcal{G}((\mathcal{F}_{\bphi}h)(\bk,t))\big)(\bx),
\end{multline*}
where
\[
\mathcal{H}_j(t,|\bk|^{\alpha},\rho_{1},\ldots,\rho_{m}) = \left\{
\begin{array}{cl}
\mathcal{K}_j^{\kappa_j}(t,|\bk|^{\alpha},\rho_{1},\ldots,\rho_{m})&\mbox{if}\quad j=0,\ldots,n_{1}-1,\\
\mathcal{K}_j(t,|\bk|^{\alpha},\rho_1,\ldots,\rho_{m})&\mbox{if}\quad  j=n_1,\ldots,n_0-1.
\end{array}\right.
\]
\item If $n_0=n_1$ and $\beta_{m}=0$ then the initial value problem \eqref{eq1Cauchynew} has a unique solution $w(\cdot,t)\in C^{n_0-1,\beta_0}[0,T]$ given by 
\begin{multline*}
w(\bx,t)=\sum_{j=0}^{n_0-1}w_j(\bx)\Psi_j(t)+\sum_{j=0}^{n_0-1}\Big(w_j*_{\bphi}\mathcal{F}_{\bphi}^{-1}\big(\mathcal{K}_j^{\kappa_j}(t,|\bk|^{\alpha},\rho_1,\ldots,\rho_{m})\big)\Big)(\bx) \\
-I_{0+}^{\beta_0,\psi}h(\bx,t)+\big(\mathcal{F}^{-1}_{\bphi}\mathcal{G}((\mathcal{F}_{\bphi}h)(\bk,t))\big)(\bx).
\end{multline*} 
\end{enumerate}
\end{thm}

\subsubsection{\bf The constant-coefficient case}\label{Cauchyproblemsconstant} 

We begin this section by recalling from \cite{luchko} the definition of the multivariate Mittag-Leffler function. This will be used in the explicit  representations of solutions of the considered fractional differential equations in the case of constant coefficients. We also refer to the entire book \cite{mittagbook} for more discussion on different types of Mittag-Leffler functions.

\begin{defn}
The multivariate Mittag-Leffler function $E_{(a_1,\ldots,a_n),b}(z_1,\ldots,z_n)$, of $n$ variables $z_1,\ldots,z_n\in\mathbb{C}$ and arbitrary parameters $a_1,\ldots,a_n,b\in\mathbb{C}$ with positive real parts, is defined by the following formula:
\begin{align*}
E_{(a_1,\ldots,a_n),b}&(z_1,\ldots,z_n)= \\
&\sum_{k=0}^{+\infty}\sum_{l_1+\cdots+l_n= k,\,\, l_1,\ldots,l_n\geqslant0}\frac{k!}{l_1!\times\cdots\times l_n!}\cdot\frac{\prod_{i=1}^n z_i^{l_i}}{\Gamma\left(b+\sum_{i=1}^n a_i l_i\right)},
\end{align*}
where this series is locally uniformly convergent under the given conditions on the parameters.
\end{defn}

Here we consider the following space-time fractional Cauchy problem with constant coefficients and both space and time derivatives taken to be fractional with respect to functions:
\begin{equation}\label{eq1CauchynewC}
\begin{split}
\prescript{C}{}\partial^{\beta_0,\psi}_{t}w(\bx,t)+\sum_{i=1}^{m-1}\rho_i \prescript{C}{}\partial^{\beta_i,\psi}_{t}w(\bx,t)+\rho_{m}\left(-\Delta\right)^{\alpha/2}_{\bphi(\bx)} w(\bx,t)&=h(\bx,t),\\
w(\bx,t)|_{_{_{t=0}}}&=w_0(\bx), \\
\frac{\partial w}{\partial\psi(t)}(\bx,t)|_{_{_{t=0}}}&=w_1(\bx),\\
&\hspace{0.2cm}\vdots \\
\frac{\partial^{n_0-1}w}{\partial\psi(t)^{n_0-1}}(\bx,t)|_{_{_{t=0}}}&=w_{n_0-1}(\bx),
\end{split}
\end{equation}
where $t\in (0,T]$,  $\bx\in\mathbb{R}^n$, $h(\cdot,t)\in C[0,T]$, $0<\alpha<2$, $\rho_i,\beta_i\in\mathbb{C}$ with $\Re (\beta_0)\not\in\mathbb{Z}$, $\rm{Re}(\beta_{\it i})>0$, $n_i=\lfloor\rm{Re}\beta_{\it i}\rfloor+1$ for $0\leqslant i\leqslant m-1$ and $\rm{Re}(\beta_0)>\rm{Re}(\beta_1)>\cdots>\rm{Re}(\beta_{\it m-1})>0$. We also suppose that $h(\bx,\cdot)\in L^1(\mathbb{R}^n,\mathrm{d}\bphi)$, $\mathcal{F}_{\bphi}h(\bx,\cdot)\in L^1(\mathbb{R}^n)$, $w_{k}\in L^1(\mathbb{R}^n,\mathrm{d}\bphi)$ and $\mathcal{F}_{\bphi}w_{k}\in L^1(\mathbb{R}^n)$ for $0\leqslant k\leqslant n_0-1$. Moreover, we also need to assume that
\begin{multline*}
\int_0^t E_{(\beta_0-\beta_1,\ldots,\beta_0-\beta_m),\beta_0}\big(-\rho_1 (\psi(t)-\psi(u))^{\beta_0-\beta_1},\ldots \big. \\
\big.\ldots,-|\bk|^{\alpha}\rho_m (\psi(t)-\psi(u))^{\beta_0-\beta_m}\big)\times \\ \times(\psi(t)-\psi(u))^{\beta_0-1}\big(\mathcal{F}_{\bphi}h(\bk,u)\big)\psi'(u)\,\mathrm{d}u\in L^1(\mathbb{R}^n),\,\,\bk\in\mathbb{R}^n,\,\,t>0.
\end{multline*}

So we have the following results about the solution of equation \eqref{eq1CauchynewC} in the space $C^{n_0-1,\beta_0}[0,T]$. The assertion is a combination of \cite[Theorems 5.6, 5.7 and 5.8.]{fractional-laplacian}. In the latter reference, it was just stated the results without a formal proof since it follows by using the same steps as in the proof of the general cases of time-variable coefficients \cite[Theorems 5.1, 5.2 and 5.3]{fractional-laplacian}. In fact, anyone can realize that the result can be obtained by applying the $n$-dimensional space Fourier transform with respect to $\bphi$ to the fractional Cauchy problem \eqref{eq1CauchynewC}, and then using \cite[Theorems 4.2, 4.3 and 4.4]{RRS} to show the solutions of equation \eqref{eq1CauchynewC} in the different cases.

\begin{thm}\label{cor3.2.1CPC}
Let $h(\cdot,t)\in C[0,T]$ and $\rho_i\in\mathbb{C}$ for $1\leqslant i\leqslant m$. Let $h(\bx,\cdot)\in L^1(\mathbb{R}^n,\mathrm{d}\bphi)$, $\mathcal{F}_{\bphi}h(\bx,\cdot)\in L^1(\mathbb{R}^n)$, $w_k\in L^1(\mathbb{R}^n,\mathrm{d}\bphi)$ and $\mathcal{F}_{\bphi}w_k\in L^1(\mathbb{R}^n)$ for $0\leqslant k\leqslant n_0-1$. 
\begin{enumerate}
    \item If $n_0>n_1$ and  $\beta_{m}=0$ then the initial value problem \eqref{eq1CauchynewC} has a unique solution $w(\cdot,t)\in C^{n_0-1,\beta_0}[0,T]$ given by
\begin{align*}
w(\bx,&t)=\sum_{j=0}^{n_0-1}w_j(\bx)\Psi_j(t) \\
&+\sum_{j=0}^{n_1-1}w_j*_{\bphi}\mathcal{F}_{\bphi}^{-1}\Bigg(\sum_{i=\kappa_j}^m \rho_i^{\star} (\psi(t)-\psi(0))^{j+\beta_0-\beta_i}\times\Bigg. \\
&\times E_{(\beta_0-\beta_1,\ldots,\beta_0-\beta_m),j+1+\beta_0-\beta_i}\big(\rho_1^{\star}(\psi(t)-\psi(0))^{\beta_0-\beta_1},\ldots,\big. \\
&\hspace{5cm}\Bigg.,\ldots,\rho_m^{\star}(\psi(t)-\psi(0))^{\beta_0-\beta_m}\big)\Bigg) \\
&+\sum_{j=n_1}^{n_0-1}w_j*_{\bphi}\mathcal{F}_{\bphi}^{-1}\Bigg(\sum_{i=0}^m \rho_i^{\star} (\psi(t)-\psi(0))^{j+\beta_0-\beta_i}\times\Bigg. \\
&\times E_{(\beta_0-\beta_1,\ldots,\beta_0-\beta_m),j+1+\beta_0-\beta_i}\big(\rho_1^{\star}(\psi(t)-\psi(0))^{\beta_0-\beta_1},\ldots,\big. \\
&\Bigg.\hspace{5cm},\ldots,\rho_m^{\star}(\psi(t)-\psi(0))^{\beta_0-\beta_m})\Bigg) \\
&+\int_0^t \Big(h(\bx,u)\Big)*_{\bphi}\mathcal{F}_{\bphi}^{-1}\Big(\psi'(u)(\psi(t)-\psi(u))^{\beta_0-1}\times\Big. \\
&\Big.\times E_{(\beta_0-\beta_1,\ldots,\beta_0-\beta_m),\beta_0}\big(-\rho_1^{\star}(\psi(t)-\psi(u))^{\beta_0-\beta_1},\ldots, \big.\\
&\Big.\hspace{5cm},\ldots,-\rho_m^{\star}(\psi(t)-\psi(u))^{\beta_0-\beta_m}\big)\Big)\,\mathrm{d}u,
\end{align*}
where $\rho_i^{\star}=\rho_i$ for $0\leqslant i\leqslant m-1$ and $\rho_m^{\star}=|\bk|^{\alpha}\rho_m$.
\item If $n_0=n_1$ and $\beta_{m}=0$ then the initial value problem \eqref{eq1CauchynewC} has a unique solution $w(\cdot,t)\in C^{n_0-1,\beta_0}[0,T]$ given by 
\begin{align*}
w(\bx,&t)=\sum_{j=0}^{n_0-1}w_j(\bx)\Psi_j(t) \\
&+\sum_{j=0}^{n_0-1}w_j*_{\bphi}\mathcal{F}_{\bphi}^{-1}\Bigg(\sum_{i=\kappa_j}^m \rho_i^{\star} (\psi(t)-\psi(0))^{j+\beta_0-\beta_i}\times\Bigg. \\
&\times E_{(\beta_0-\beta_1,\ldots,\beta_0-\beta_m),j+1+\beta_0-\beta_i}\big(\rho_1^{\star}(\psi(t)-\psi(0))^{\beta_0-\beta_1},\ldots, \\
&\Bigg.\hspace{5cm},\ldots,\rho_m^{\star}(\psi(t)-\psi(0))^{\beta_0-\beta_m}\big)\Bigg) \\
&+\int_0^t \Big(h(\bx,u)\Big)*_{\bphi}\mathcal{F}_{\bphi}^{-1}\Big(\psi'(u)(\psi(t)-\psi(u))^{\beta_0-1}\times\Big. \\
&\times E_{(\beta_0-\beta_1,\ldots,\beta_0-\beta_m),\beta_0}\big(-\rho_1^{\star}(\psi(t)-\psi(u))^{\beta_0-\beta_1},\ldots, \\
&\Big.\hspace{5cm},\ldots,-\rho_m^{\star}(\psi(t)-\psi(u))^{\beta_0-\beta_m}\big)\Big)\,\mathrm{d}u,
\end{align*}
where $\rho_i^{\star}=\rho_i$ for $0\leqslant i\leqslant m-1$ and $\rho_m^{\star}=|\bk|^{\alpha}\rho_m$.
\end{enumerate}
\end{thm}

\subsection{Clifford Analysis} Here we recall some necessary facts and terminology on the Clifford analysis. Nevertheless, for more details on this topic, see e.g. \cite{bookclifford1}. Let us start by recalling the universal real Clifford algebra. We then take the $n$-dimensional vector space $\mathbb{R}^n$ endowed with an orthonormal basis $\{e_1,\ldots,e_n\}$. The universal real Clifford algebra $Cl_{0,n}$ is defined as the $2^n$-dimensional associative algebra which satisfies the following multiplication rule 
\[e_{i}e_{j}+e_{j}e_{i}=-2\delta_{ij},\quad i,j=1,\ldots,n.\]
A vector space basis for $Cl_{0,n}$ is generated by the elements $e_0=1$ and $e_{B}=e_{r_1,\ldots,r_k}$, where $B=\{r_1,\ldots,r_k\}\subset N=\{1,\ldots,n\}$ for $1\leqslant r_1<\cdots<r_k\leqslant n$. Hence, for any $y\in Cl_{0,n}$ we have that  $y=\sum_{B}x_{B}e_{B}$ with $x_B\in\mathbb{R}$. Now we recall the complexified Clifford algebra $\mathbb{C}_n$:
\[ \mathbb{C}_n =\mathbb{C}\otimes Cl_{0,n}=\left\{v=\sum_{B}v_B e_{B},\,\,v_B\in\mathbb{C},\,\,B\subset{N}\right\}, \] 
where the imaginary unit $i$ of $\mathbb{C}$ commutes with the basis elements $(ie_j =e_j i$ for any $j=1,\ldots,n)$. A $\mathbb{C}_n$-valued function defined on an open subset $V\subset\mathbb{R}^n$ can be represented by $f=\sum_{B}f_B e_B$ with $\mathbb{C}$-valued components $f_B$. As usual, the continuity, differentiability and other properties are normally assumed component-wisely by means of the classical notions on $\mathbb{C}$. 

For the next definition we need to recall the Euclidean Dirac operator 
\[
D_{\bx}=\sum_{k=1}^n e_k \partial_{\bx_k},\quad\text{and it has the property}\quad D_{\bx}^2=-\Delta=-\sum_{k=1}^n \partial_{\bx_k}^2.       
\]
\begin{defn}\cite[Chapter 2]{bookclifford1}
A Clifford valued $C^1$ function $f$ is left-monogenic if $D_{\bx}f=0$ on $V$, respectively right-monogenic if $fD_{\bx}=0$ on $V$.
\end{defn}

The above definition will be used implicitly in Section \ref{special} to illustrate some particular cases of the main results of the present paper.

In the next section we will introduce a new class of generalized fractional Dirac type operators. Hence, we need to use and describe a Witt basis. Let us embed $\mathbb{R}^n$ into $\mathbb{R}^{n+2}$ by considering two new elements $e_+$ and $e_{-}$ which satisfy $e_+^2=1$, $e_{-}^2=-1$ and $e_{+}e_{-}+e_{-}e_+=0$. We also suppose that $e_{-},e_+$ anti-commute with each element from $\{e_1,\ldots,e_n\}$. Then $\{e_1,\ldots,e_n,e_+,e_{-}\}$ spans $\mathbb{R}^{n+1,1}$. By using the elements $e_{+},e_-$ we compose two nilpotent elements usually denoted by $\mathfrak{f}$ and $\mathfrak{f}^+$. They are defined by:
\[\mathfrak{f}=\frac{e_+-e_-}{2}\quad\text{and}\quad \mathfrak{f}^+=\frac{e_++e_-}{2}.\]   
Some useful properties:
\begin{enumerate}
\item\label{uno1} $(\mathfrak{f})^2= (\mathfrak{f}^+)^2=0,$
\item\label{dos2} $\mathfrak{f}\mathfrak{f}^++\mathfrak{f}^+\mathfrak{f}=1,$
\item\label{tres3} $\mathfrak{f}e_i+e_i\mathfrak{f}=\mathfrak{f}^+ e_i+e_i\mathfrak{f}^+=0,\quad i=1,\ldots,n.$
\end{enumerate}

\section{Main results}\label{mainresults}

In this section, we study some general fractional Cauchy type problems by using generalized fractional Dirac type operators, which are combinations of the (fractional) Laplacian and derivatives with respect to another function. We show in all cases the explicit solutions.

\subsection{Generalized fractional Dirac type operators with time variable coefficients}

By using the Witt basis $\{e_1,\ldots,e_n,\mathfrak{f},\mathfrak{f}^+\}$ we formally introduce a new class of generalized fractional Dirac type operators with time variable coefficients and with respect to the given functions $\bphi,\psi$ by
\begin{equation}\label{classdiracnew}
_{\bphi(\bx),t}D_{\rho_{1},\ldots,\rho_{m};\,\psi}^{\lambda,\beta_0,\ldots,\beta_{m-1}}:=\rho_{m}^{1/2}(t)(-\Delta)^{\lambda/2}_{\bphi(\bx)}+\mathfrak{f}\left(^{C}\partial_{t}^{\beta_0,\psi}+\sum_{i=1}^{m-1}\rho_i(t) ^{C}\partial_{t}^{\beta_i,\psi}\right)+\mathfrak{f}^+,
\end{equation}
where $\bx\in\mathbb{R}^n$, $t>0$, $0<\lambda<1$, $\beta_i\in\mathbb{C}$, $\rm{Re}(\beta_0)>\rm{Re}(\beta_1)>\cdots>\rm{Re}(\beta_{{\it m}-1})>0$, $\Re(\beta_0)\not\in\mathbb{Z}$ and $n_i=\lfloor \rm{Re}\beta_{\it i}\rfloor+1$ for $i=0,1,\ldots,m-1$. We also assume that $\rho_i(t)\in C[0,T]$ $(i=1,\ldots,m)$ and $\bphi:\mathbb{R}^n\to\mathbb{R}^n$ is a bijection all of whose $1$st-order partial derivatives exist almost everywhere. 

In a very particular case, notice that the generalized fractional Dirac type operator of \eqref{classdiracnew} becomes the one introduced in \cite[Formula (3.1)]{BRS} when $\psi(t)=t$ and $\bphi(\bx)=\bx$.

\begin{prop}
Let $\bx\in\mathbb{R}^n$, $t>0$, $0<\lambda<1$, $\beta_i\in\mathbb{C}$, $\rm{Re}(\beta_0)>\rm{Re}(\beta_1)>\cdots>\rm{Re}(\beta_{{\it m}-1})>0$, $\Re(\beta_0)\not\in\mathbb{Z}$ and $n_i=\lfloor \rm{Re}\beta_{\it i}\rfloor+1$, $i=0,1,\ldots,m-1$. We also suppose that $\rho_i(t)\in C[0,T]$ $(i=1,\ldots,m)$ and $\bphi:\mathbb{R}^n\to\mathbb{R}^n$ is a bijection all of whose $1$st-order partial derivatives exist almost everywhere. The following factorization holds:
\begin{equation}\label{factornew} 
\big(_{\bphi(\bx),t}D_{\rho_{1},\ldots,\rho_{m};\,\psi}^{\lambda,\beta_0,\ldots,\beta_{m-1}}\big)^2=\rho_{m}(t)(-\Delta)^{\lambda}_{\bphi(\bx)}+\,^{C}\partial_{t}^{\beta_0,\psi}+\sum_{i=1}^{m-1}\rho_i(t) ^{C}\partial_{t}^{\beta_i,\psi}.
\end{equation}
\end{prop}
\begin{proof}
We know that
\begin{align*}
&(\,_{\bphi(\bx),t}D_{\rho_{1},\ldots,\rho_{m};\,\psi}^{\lambda,\beta_0,\ldots,\beta_{m-1}})^2= \\
&\left(\rho_{m}^{1/2}(t)(-\Delta)^{\lambda/2}_{\bphi(\bx)}+\mathfrak{f}\left(\prescript{C}{}\partial_{t}^{\beta_0,\psi}+\sum_{i=1}^{m-1}\rho_i(t) ^{C}\partial_{t}^{\beta_i,\psi}\right)+\mathfrak{f}^+\right)^2.
\end{align*}
Notice that for  
\[E=\rho_{m}^{1/2}(t)(-\Delta)^{\lambda/2}_{\bphi(\bx)},\quad F=\,^{C}\partial_{t}^{\beta_0,\psi}+\sum_{i=1}^{m-1}\rho_i(t) ^{C}\partial_{t}^{\beta_i,\psi},\]
it follows that
\[ (\,_{\bphi(\bx),t}D_{\rho_{1},\ldots,\rho_{m};\,\phi}^{\lambda,\beta_0,\ldots,\beta_{m-1}})^2=\big(E+\mathfrak{f}F+\mathfrak{f}^+\big)\big(E+\mathfrak{f}F+\mathfrak{f}^+\big). \]
By Theorem \ref{Thm:MFLn:conjug} we have 
\begin{align*}
\left(-\Delta\right)^{\lambda/2}_{\bphi(\bx)}\circ \left(-\Delta\right)^{\lambda/2}_{\bphi(\bx)}&=\big(Q_{\bphi}\circ\left(-\Delta\right)^{\lambda/2}\circ Q_{\bphi}^{-1}\big)\circ \big(Q_{\bphi}\circ\left(-\Delta\right)^{\lambda/2}\circ Q_{\bphi}^{-1}\big) \\
&=Q_{\bphi}\circ\left(-\Delta\right)^{\lambda/2}\circ \left(-\Delta\right)^{\lambda/2}\circ Q_{\bphi}^{-1} \\
&=Q_{\bphi}\circ\left(-\Delta\right)^{\lambda}\circ Q_{\bphi}^{-1}=\left(-\Delta\right)^{\lambda}_{\bphi(\bx)}.
\end{align*}
By the properties \eqref{uno1}, \eqref{dos2}, \eqref{tres3} we obtain:
\begin{align*}
\big(E+\mathfrak{f}F+\mathfrak{f}^+\big)&\big(E+\mathfrak{f}F+\mathfrak{f}^+\big) \\
&=EE-\mathfrak{f}EF-\mathfrak{f}^+ \mathfrak{f}EF+(\mathfrak{f})^2 FF+\mathfrak{f}\mathfrak{f}^+ F+\mathfrak{f}^+ E+\mathfrak{f}^+ \mathfrak{f}F+(\mathfrak{f})^2 \\
&=EE+\mathfrak{f}\mathfrak{f}^+ F+\mathfrak{f}^+ \mathfrak{f}F=EE+(\mathfrak{f}\mathfrak{f}^+ +\mathfrak{f}^+ \mathfrak{f})F=EE+F,
\end{align*}
which complete the proof.
\end{proof}

\subsection{Explicit solution of fractional Cauchy type problems with time variable  coefficients}

Now we give the main results of the paper. Below we just treat the case $n_0=n_1$ and $\beta_{m}=0$. Other cases are similar. 

\begin{thm}\label{cliffordnew}
Let $n_0=n_1$, $\beta_{m}=0$, $0<\lambda<1$, $\beta_i\in\mathbb{C}$, $\rm{Re}(\beta_0)>\rm{Re}(\beta_1)>\cdots>\rm{Re}(\beta_{{\it m}-1})>0$, $\Re(\beta_0)\not\in\mathbb{Z}$ and $n_i=\lfloor \rm{Re}\beta_{\it i}\rfloor+1$ for $i=0,1,\ldots,m-1$. We also suppose that $\rho_i(t)\in C[0,T]$ $(i=1,\ldots,m)$ and $\bphi:\mathbb{R}^n\to\mathbb{R}^n$ is a bijection all of whose $1$st-order partial derivatives exist almost everywhere. The following fractional Cauchy type problem 
\begin{equation}\label{eq1Cauchynewclinew}
\begin{split}
\left(\rho_{m}^{1/2}(t)(-\Delta)^{\lambda/2}_{\bphi(\bx)}+\mathfrak{f}\left(^{C}\partial_{t}^{\beta_0,\psi}+\sum_{i=1}^{m-1}\rho_i(t) ^{C}\partial_{t}^{\beta_i,\psi}\right)+\mathfrak{f}^+\right)w(\bx,t)&=0,\\
w(\bx,t)|_{_{_{t=0}}}&=r_0(\bx),  \\
\frac{\partial w}{\partial\psi(t)}(\bx,t)|_{_{_{t=0}}}&=r_1(\bx),\\
&\hspace{0.2cm}\vdots \\
\frac{\partial^{n_0-1}w}{\partial\psi(t)^{n_0-1}}(\bx,t)|_{_{_{t=0}}}&=r_{n_0-1}(\bx),
\end{split}
\end{equation}
is solvable for any $\bx\in\mathbb{R}^n$ and $t\in (0,T]$, and the solution is given by
\begin{align*}
&w(\bx,t)=\sum_{j=0}^{n_0-1}\rho_{m}^{1/2}(t)\Psi_j(t)(-\Delta)^{\lambda/2}_{\bphi(\bx)} \big(r_j(\bx)\big) \\
&+\rho_{m}^{1/2}(t)\sum_{j=0}^{n_0-1}\int_{\mathbb{R}^n}(-\Delta)^{\lambda/2}_{\bphi(\bx)}r_j\big(\bphi^{-1}\big(\bphi(\bx)-\bphi(\by)\big)\big)\times \nonumber\\ 
&\hspace{3cm}\times \mathcal{F}_{\bphi}^{-1}\big(\mathcal{K}_j^{\kappa_j}(t,|\by|^{\lambda},\rho_1,\ldots,\rho_{m})\big)\big|J\bphi(\by)\big|\,\mathrm{d}\by\nonumber \\
&+\mathfrak{f}\left(\sum_{i=1}^{m-1}\sum_{j=0}^{n_0-1}\rho_i(t)\,^{C}\partial_{t}^{\beta_i,\psi}\Psi_j(t) r_j(\bx)\right.\nonumber\\
&+\sum_{j=0}^{n_0-1}\Big(r_j*_{\bphi}\mathcal{F}_{\bphi}^{-1}\big(\prescript{C}{}\partial_t^{\beta_0,\psi}\mathcal{K}_j^{\kappa_j}(t,|\bk|^{\lambda},\rho_1,\ldots,\rho_{m})\big)\Big)(\bx) \nonumber\\
&\left.+\sum_{i=1}^{m-1}\sum_{j=0}^{n_0-1}\rho_i(t)\Big(r_j*_{\bphi}\mathcal{F}_{\bphi}^{-1}\big(\prescript{C}{}\partial_t^{\beta_i,\psi}\mathcal{K}_j^{\kappa_j}(t,|\bk|^{\lambda},\rho_1,\ldots,\rho_{m})\big)\Big)(\bx)\right)\nonumber \\
&+\mathfrak{f}^{+}\left(\sum_{j=0}^{n_0-1}r_j(\bx)\Psi_j(t)+\sum_{j=0}^{n_0-1}\Big(r_j*_{\bphi}\mathcal{F}_{\bphi}^{-1}\big(\mathcal{K}_j^{\kappa_j}(t,|\bk|^{\lambda},\rho_1,\ldots,\rho_{m})\big)\Big)(\bx)\right),\nonumber 
\end{align*}
where $\Psi_j(t)$ is that of \eqref{PSI}, $\mathcal{K}_j^{\kappa_j}$ is the same of \eqref{kjkappa} and $\varkappa_j=\displaystyle{\min\{\mathbb{K}_j\}}$.
\end{thm}
\begin{proof}
Notice first that equation \eqref{eq1Cauchynewclinew} is equivalent to 
\begin{equation}\label{caenew}
_{\bphi(\bx),t}D_{\rho_{1},\ldots,\rho_{m};\,\psi}^{\lambda,\beta_0,\ldots,\beta_{m-1}}w(\bx,t)=0,\quad \bx\in\mathbb{R}^n,\,\,t>0.
\end{equation}
Applying the operator $_{\bphi(\bx),t}D_{\rho_{1},\ldots,\rho_{m};\,\psi}^{\lambda,\beta_0,\ldots,\beta_{m-1}}$ to \eqref{caenew} implies that
\[\rho_{m}(t)(-\Delta)^{\lambda}_{\bphi(\bx)}w(\bx,t)+\prescript{C}{}\partial_{t}^{\beta_0,\psi}w(\bx,t)+\sum_{i=1}^{m-1}\rho_i(t)\prescript{C}{}\partial_{t}^{\beta_i,\psi}w(\bx,t)=0,\]
due to the factorization \eqref{factornew}. We then obtain the equation \eqref{eq1Cauchynew} with $h\equiv0$. Thus, by Theorem \ref{cor3.2.1CP}, the solution of equation \eqref{eq1Cauchynewclinew} is given by 
\[
_{\bphi(\bx),t}D_{\rho_{1},\ldots,\rho_{m};\,\psi}^{\lambda,\beta_0,\ldots,\beta_{m-1}}w(\bx,t),
\]
where 
\[
w(\bx,t)=\sum_{j=0}^{n_0-1}r_j(\bx)\Psi_j(t)+\sum_{j=0}^{n_0-1}\Big(r_j*_{\bphi}\mathcal{F}_{\bphi}^{-1}\big(\mathcal{K}_j^{\kappa_j}(t,|\bk|^{\lambda},\rho_1,\ldots,\rho_{m})\big)\Big)(\bx).
\]
The explicit representation of the solution follows by calculating each of the components of $_{\bphi(\bx),t}D_{\rho_{1},\ldots,\rho_{m};\,\psi}^{\lambda,\beta_0,\ldots,\beta_{m-1}}w(\bx,t)$, separately. In fact, we have
\begin{align*}
&_{\bphi(\bx),t}D_{\rho_{1},\ldots,\rho_{m};\,\psi}^{\lambda,\beta_0,\ldots,\beta_{m-1}}w(\bx,t)= \\
&\rho_{m}^{1/2}(t)(-\Delta)^{\lambda/2}_{\bphi(\bx)} w(\bx,t)+\mathfrak{f}\left(^{C}\partial_{t}^{\beta_0,\psi}+\sum_{i=1}^{m-1}\rho_i(t) ^{C}\partial_{t}^{\beta_i,\psi}\right)w(\bx,t)+\mathfrak{f}^{+}w(\bx,t).
\end{align*}
We first get
\begin{align*}
\rho_{m}^{1/2}(t)&(-\Delta)^{\lambda/2}_{\bphi(\bx)}w(\bx,t)=\sum_{j=0}^{n_0-1}\rho_{m}^{1/2}(t)(-\Delta)^{\lambda/2}_{\bphi(\bx)} \big(r_j(\bx)\big)\Psi_j(t) \\
&+\rho_{m}^{1/2}(t)\sum_{j=0}^{n_0-1}\int_{\mathbb{R}^n}(-\Delta)^{\lambda/2}_{\bphi(\bx)}r_j\big(\bphi^{-1}\big(\bphi(\bx)-\bphi(\by)\big)\big)\times \\ &\hspace{3cm}\times \mathcal{F}_{\bphi}^{-1}\big(\mathcal{K}_j^{\kappa_j}(t,|\by|^{\lambda},\rho_1,\ldots,\rho_{m})\big)\big|J\bphi(\by)\big|\,\mathrm{d}\by.
\end{align*}
We also have
\begin{align*}
&\left(^{C}\partial_{t}^{\beta_0,\psi}+\sum_{i=1}^{m-1}\rho_i(t) ^{C}\partial_{t}^{\beta_i,\psi}\right)w(\bx,t)=\sum_{i=1}^{m-1}\sum_{j=0}^{n_0-1}\rho_i(t)r_j(\bx)\,^{C}\partial_{t}^{\beta_i,\psi}\Psi_j(t)\\
&+\sum_{j=0}^{n_0-1}\Big(r_j*_{\bphi}\mathcal{F}_{\bphi}^{-1}\big(\prescript{C}{}\partial_t^{\beta_0,\psi}\mathcal{K}_j^{\kappa_j}(t,|\bk|^{\lambda},\rho_1,\ldots,\rho_{m})\big)\Big)(\bx) \\
&+\sum_{i=1}^{m-1}\sum_{j=0}^{n_0-1}\rho_i(t)\Big(r_j*_{\bphi}\mathcal{F}_{\bphi}^{-1}\big(\prescript{C}{}\partial_t^{\beta_i,\psi}\mathcal{K}_j^{\kappa_j}(t,|\bk|^{\lambda},\rho_1,\ldots,\rho_{m})\big)\Big)(\bx),
\end{align*}
since $\,^{C}\partial_{t}^{\beta_0,\phi}\Psi_j(t)=0$ for any $j=0,1,\ldots,n_0-1$.
\end{proof}

\begin{rem}
It is important to note that the fractional parameter in the Laplacian operator of equation \eqref{eq1Cauchynewclinew} is just defined in the range $\lambda/2\in(0,1/2)$. This anomaly is given by the consideration of the equation \eqref{eq1Cauchynewclinew} in the new basis of Witt type.  
\end{rem}

The following statement follows by Theorem \ref{cliffordnew} with $\bphi(\bx)=\bx$.  

\begin{cor}\label{cliffordnew2}
Let $n_0=n_1$, $\beta_{m}=0$, $0<\lambda<1$, $\beta_i\in\mathbb{C}$, $\rm{Re}(\beta_0)>\rm{Re}(\beta_1)>\cdots>\rm{Re}(\beta_{{\it m}-1})>0$, $\Re(\beta_0)\not\in\mathbb{Z}$ and $n_i=\lfloor \rm{Re}\beta_{\it i}\rfloor+1$ for $i=0,1,\ldots,m-1$. Suppose also that $\rho_i(t)\in C[0,T]$ $(i=1,\ldots,m)$. The following fractional Cauchy type problem 
\begin{equation*}
\begin{split}
\left(\rho_{m}^{1/2}(t)(-\Delta)^{\lambda/2}_{\bx}+\mathfrak{f}\left(^{C}\partial_{t}^{\beta_0,\psi}+\sum_{i=1}^{m-1}\rho_i(t) ^{C}\partial_{t}^{\beta_i,\psi}\right)+\mathfrak{f}^+\right)w(\bx,t)&=0,\\
w(\bx,t)|_{_{_{t=0}}}&=r_0(\bx),  \\
\frac{\partial w}{\partial\psi(t)}(\bx,t)|_{_{_{t=0}}}&=r_1(\bx),\\
&\hspace{0.2cm}\vdots \\
\frac{\partial^{n_0-1}w}{\partial\psi(t)^{n_0-1}}(\bx,t)|_{_{_{t=0}}}&=r_{n_0-1}(\bx),
\end{split}
\end{equation*}
is solvable for any $\bx\in\mathbb{R}^n$ and $t\in (0,T]$, and the solution is given by
\begin{align*}
w&(\bx,t)=\sum_{j=0}^{n_0-1}\rho_{m}^{1/2}(t)\Psi_j(t)(-\Delta)^{\lambda/2}_{\bx} \big(r_j(\bx)\big) \\
&+\rho_{m}^{1/2}(t)\sum_{j=0}^{n_0-1}\int_{\mathbb{R}^n}(-\Delta)^{\lambda/2}_{\bx}r_j\big(\bx-\by\big)\big)\mathcal{F}^{-1}\big(\mathcal{K}_j^{\kappa_j}(t,|\by|^{\lambda},\rho_1,\ldots,\rho_{m})\big)\,\mathrm{d}\by\nonumber \\
&+\mathfrak{f}\left(\sum_{i=1}^{m-1}\sum_{j=0}^{n_0-1}\rho_i(t)\,^{C}\partial_{t}^{\beta_i,\psi}\Psi_j(t) r_j(\bx)\right.\nonumber\\
&+\sum_{j=0}^{n_0-1}\Big(r_j*\mathcal{F}^{-1}\big(\prescript{C}{}\partial_t^{\beta_0,\psi}\mathcal{K}_j^{\kappa_j}(t,|\bk|^{\lambda},\rho_1,\ldots,\rho_{m})\big)\Big)(\bx) \nonumber\\
&\left.+\sum_{i=1}^{m-1}\sum_{j=0}^{n_0-1}\rho_i(t)\Big(r_j*\mathcal{F}^{-1}\big(\prescript{C}{}\partial_t^{\beta_i,\psi}\mathcal{K}_j^{\kappa_j}(t,|\bk|^{\lambda},\rho_1,\ldots,\rho_{m})\big)\Big)(\bx)\right)\nonumber \\
&+\mathfrak{f}^{+}\left(\sum_{j=0}^{n_0-1}r_j(\bx)\Psi_j(t)+\sum_{j=0}^{n_0-1}\Big(r_j*\mathcal{F}^{-1}\big(\mathcal{K}_j^{\kappa_j}(t,|\bk|^{\lambda},\rho_1,\ldots,\rho_{m})\big)\Big)(\bx)\right),\nonumber 
\end{align*}
where $\Psi_j(t)$ is that of \eqref{PSI}, $\mathcal{K}_j^{\kappa_j}$ is the same of \eqref{kjkappa} and $\varkappa_j=\displaystyle{\min\{\mathbb{K}_j\}}$.
\end{cor}

\subsection{ Explicit solution of fractional Cauchy type problems with constant coefficients}

We show that the solution can be given in a more explicit form by using the multivariate Mittag-Leffler function. Again, we just consider the case $n_0=n_1$ and $\beta_{m}=0$. Other cases can prove similarly. Instead of transforming the solution given in Theorem \ref{cliffordnew}, we give a complete proof by using the result of Theorem \ref{cor3.2.1CPC}.   

\begin{thm}\label{constant-thm}
Let $n_0=n_1$, $\beta_{m}=0$, $0<\lambda<1$, $\beta_i,\rho_i\in\mathbb{C}$, $\rm{Re}(\beta_0)>\rm{Re}(\beta_1)>\cdots>\rm{Re}(\beta_{{\it m}-1})>0$, $\Re(\beta_0)\not\in\mathbb{Z}$ and $n_i=\lfloor \rm{Re}\beta_{\it i}\rfloor+1$ for $i=0,1,\ldots,m$. Let $\bphi:\mathbb{R}^n\to\mathbb{R}^n$ be a bijection all of whose $1$st-order partial derivatives exist almost everywhere. The fractional Cauchy type problem \eqref{eq1Cauchynewclinew} with constant coefficients $(\rho_i(t)=\rho_i\in\mathbb{C},\,i=1,\ldots,m)$ is solvable for any $\bx\in\mathbb{R}^n$ and $t\in (0,T]$, and the solution is given by
\begin{align*}
&w(\bx,t)=\sum_{j=0}^{n_0-1}\rho_{m}^{1/2}\Psi_j(t)(-\Delta)^{\lambda/2}_{\bphi(\bx)} \big(r_j(\bx)\big) \\
&+\rho_{m}^{1/2}\sum_{j=0}^{n_0-1}\int_{\mathbb{R}^n}(-\Delta)^{\lambda/2}_{\bphi(\bx)}r_j\big(\bphi^{-1}\big(\bphi(\bx)-\bphi(\by)\big)\big)\times \\
&\times\mathcal{F}_{\bphi}^{-1}\Bigg(\sum_{i=\kappa_j}^m \rho_i^{\star} (\psi(t)-\psi(0))^{j+\beta_0-\beta_i}\times\Bigg. \\
&\hspace{1cm}\times E_{(\beta_0-\beta_1,\ldots,\beta_0-\beta_m),j+1+\beta_0-\beta_i}\big(\rho_1(\psi(t)-\psi(0))^{\beta_0-\beta_1},\ldots,\big. \\
&\hspace{0.5cm}\big.\Bigg.\ldots,\rho_{m-1}(\psi(t)-\psi(0))^{\beta_0-\beta_{m-1}},\rho_m |\by|^{\lambda}(\psi(t)-\psi(0))^{\beta_0-\beta_m}\big)\Bigg)\big|J\bphi(\by)\big|\,\mathrm{d}\by \\
&+\mathfrak{f}\left(\sum_{i=1}^{m-1}\sum_{j=0}^{n_0-1}\rho_i\,^{C}\partial_{t}^{\beta_i,\psi}\Psi_j(t) r_j(\bx)\right.\\
&\hspace{1cm}+\sum_{j=0}^{n_0-1}w_j*_{\bphi}\mathcal{F}_{\bphi}^{-1}\left(\prescript{C}{}\partial^{\beta_0,\psi}_{t}\sum_{i=\kappa_j}^m \rho_i^{\star} (\psi(t)-\psi(0))^{j+\beta_0-\beta_i}\times\right. \\
&\hspace{2cm}\times E_{(\beta_0-\beta_1,\ldots,\beta_0-\beta_m),j+1+\beta_0-\beta_i}\big(\rho_1^{\star}(\psi(t)-\psi(0))^{\beta_0-\beta_1},\ldots,\big. \\
&\big.\Bigg.\hspace{7cm},\ldots,\rho_m^{\star}(\psi(t)-\psi(0))^{\beta_0-\beta_m}\big)\Bigg) \\
&+\sum_{i=1}^{m-1}\sum_{j=0}^{n_0-1}\rho_i w_j*_{\bphi} \mathcal{F}_{\bphi}^{-1}\left(\prescript{C}{}\partial^{\beta_i,\psi}_{t}\sum_{i=\kappa_j}^m \rho_i^{\star} (\psi(t)-\psi(0))^{j+\beta_0-\beta_i}\times\right.\\
&\hspace{1cm}\times E_{(\beta_0-\beta_1,\ldots,\beta_0-\beta_m),j+1+\beta_0-\beta_i}\big(\rho_1^{\star}(\psi(t)-\psi(0))^{\beta_0-\beta_1},\ldots,\big. \\
&
.\left.\big.\hspace{7cm},\ldots,\rho_m^{\star}(\psi(t)-\psi(0))^{\beta_0-\beta_m}\big)\right) \\
&\left.\left.+\sum_{i=1}^{m-1}\sum_{j=0}^{n_0-1}\rho_i\Big(r_j*_{\bphi}\mathcal{F}_{\bphi}^{-1}\big(\prescript{C}{}\partial_t^{\beta_i,\psi}\mathcal{K}_j^{\kappa_j}(t,|\bk|^{\lambda},\rho_1,\ldots,\rho_{m})\big)\Big)(x)\right)\right) \\
&+\mathfrak{f}^{+}\left(\sum_{j=0}^{n_0-1}w_j(\bx)\Psi_j(t)+\sum_{j=0}^{n_0-1}w_j*_{\bphi}\mathcal{F}_{\bphi}^{-1}\Bigg(\sum_{i=\kappa_j}^m \rho_i^{\star} (\psi(t)-\psi(0))^{j+\beta_0-\beta_i}\times\right.\Bigg. \\
&\hspace{1cm}\times E_{(\beta_0-\beta_1,\ldots,\beta_0-\beta_m),j+1+\beta_0-\beta_i}\big(\rho_1^{\star}(\psi(t)-\psi(0))^{\beta_0-\beta_1},\ldots,\big. \\
&\hspace{6.5cm}\Bigg.\left.\big.,\ldots,\rho_m^{\star}(\psi(t)-\psi(0))^{\beta_0-\beta_m}\big)\Bigg)\right),
\end{align*}
where $\Psi_j(t)$ is that of \eqref{PSI}, $\rho_i^{\star}=\rho_i$ for $0\leqslant i\leqslant m-1$ and $\rho_m^{\star}=|\bk|^{\lambda}\rho_m$.
\end{thm}
\begin{proof}
Equation \eqref{eq1Cauchynewclinew} can be written as  
\begin{equation}\label{caenew-const}
_{\bphi(\bx),t}D_{\rho_{1},\ldots,\rho_{m};\,\psi}^{\lambda,\beta_0,\ldots,\beta_{m-1}}w(\bx,t)=0,\quad \bx\in\mathbb{R}^n,\,\,t>0,
\end{equation}
where $\rho_i\in\mathbb{C}$ for $i=1,\ldots,m.$ By \eqref{factornew} and the application of $_{\bphi(\bx),t}D_{\rho_{1},\ldots,\rho_{m};\,\psi}^{\lambda,\beta_0,\ldots,\beta_{m-1}}$ to \eqref{caenew-const} we obtain
\[\rho_{m}(-\Delta)^{\lambda}_{\bphi(\bx)}w(\bx,t)+\prescript{C}{}\partial_{t}^{\beta_0,\psi}w(\bx,t)+\sum_{i=1}^{m-1}\rho_i\prescript{C}{}\partial_{t}^{\beta_i,\psi}w(\bx,t)=0.\]
The latter equation is the same of \eqref{eq1CauchynewC} with $h\equiv0$. Hence, by Theorem \ref{cor3.2.1CPC}, the solution of equation \eqref{eq1Cauchynewclinew} ($\rho_i\in\mathbb{C}$ for $i=1,\ldots,m$) is given by 
\[
_{\bphi(\bx),t}D_{\rho_{1},\ldots,\rho_{m};\,\psi}^{\lambda,\beta_0,\ldots,\beta_{m-1}}w(\bx,t),
\]
where 
\begin{align*}
w(\bx,t)&=\sum_{j=0}^{n_0-1}w_j(\bx)\Psi_j(t)+\sum_{j=0}^{n_0-1}w_j*_{\bphi}\mathcal{F}_{\bphi}^{-1}\Bigg(\sum_{i=\kappa_j}^m \rho_i^{\star} (\psi(t)-\psi(0))^{j+\beta_0-\beta_i}\times\Bigg. \\
&\hspace{1cm}\times E_{(\beta_0-\beta_1,\ldots,\beta_0-\beta_m),j+1+\beta_0-\beta_i}\big(\rho_1^{\star}(\psi(t)-\psi(0))^{\beta_0-\beta_1},\ldots,\big. \\
&\hspace{5.5cm}\big.\Bigg.,\ldots,\rho_m^{\star}(\psi(t)-\psi(0))^{\beta_0-\beta_m}\big)\Bigg) 
\end{align*}
where $\rho_i^{\star}=\rho_i$ for $0\leqslant i\leqslant m-1$ and $\rho_m^{\star}=|\bk|^{\lambda}\rho_m$. Hence
\begin{align*}
_{\bphi(\bx),t}&D_{\rho_{1},\ldots,\rho_{m};\,\psi}^{\lambda,\beta_0,\ldots,\beta_{m-1}}w(\bx,t)= \\
&\rho_{m}^{1/2}(-\Delta)^{\lambda/2}_{\bphi(\bx)} w(\bx,t)+\mathfrak{f}\left(^{C}\partial_{t}^{\beta_0,\psi}+\sum_{i=1}^{m-1}\rho_i\,^{C}\partial_{t}^{\beta_i,\psi}\right)w(\bx,t)+\mathfrak{f}^{+}w(\bx,t).
\end{align*}
It remains to calculate each component of the expression above, which is a matter of doing some estimates.
\end{proof}

Relative simple expressions for the solution of equation \eqref{eq1Cauchynewclinew} with constant coefficients $(\rho_i(t)=\rho_i\in\mathbb{C},\,i=1,\ldots,m)$ can be straightforward obtained by setting $\bphi(\bx)=\bx$ or, $\psi(t)=t$ or $\bphi(\bx)=\bx$ and $\psi(t)=t$. 

\subsection{ Examples}\label{examples}
 
In the next two examples we suppose that $\psi(t)=t$ for any $t\in[0,T]$ and we denote $\prescript{C}{}\partial_{t}^{\beta_0,\psi}$, $I_{0+}^{\beta_0,\psi}$ simply by $\prescript{C}{}\partial_{t}^{\beta_0}$, $I_{0+}^{\beta_0}.$ 
 
\subsubsection{\bf Wave and heat type equations}

Let us consider the fractional initial value problem of wave type with a power time-function:

\begin{equation}\label{phiexamplewave}
\left\{ \begin{split}
\left(t^{\beta_0/2}(-\Delta)^{\alpha/4}_{\bphi(\bx)}+\prescript{C}{}\partial_{t}^{\beta_0}\mathfrak{f}+\mathfrak{f}^+\right)v(\bx,t)&=0,\\
v(\bx,t)|_{_{_{t=0}}}&=v_0(\bx),  \\
\partial_t v(\bx,t)|_{_{_{t=0}}}&=v_1(\bx)
\end{split}
\right.
\end{equation}
where $v_0,v_1\in L^1(\mathbb{R}^n,\mathrm{d}\bphi)$, $\mathcal{F}_{\bphi}v_0,\mathcal{F}_{\bphi}v_1\in L^1(\mathbb{R}^n)$, $1<\beta_0<2$ and $0<\alpha<2$. By Theorem \ref{cliffordnew} we know that the solution of equation \eqref{phiexamplewave} is given by 
\[
\left(t^{\beta_0/2}(-\Delta)^{\alpha/4}_{\bphi(\bx)}+\prescript{C}{}\partial_{t}^{\beta_0}\mathfrak{f}+\mathfrak{f}^+\right)w(\bx,t)
\] 
where $w(\bx,t)$ is the solution of the following equation:
\begin{equation}\label{example1}
\left\{
\begin{split}
\prescript{C}{}\partial_{t}^{\beta_0}w(\bx,t)+t^{\beta_0}(-\Delta)^{\alpha/2}_{\bphi(\bx)} w(\bx,t)&=0, \\
w(\bx,0)&=w_0(\bx), \\
\frac{\partial w}{\partial t}(\bx,0)&=w_1(\bx),
\end{split}\right.
\end{equation}
where $w_0,w_1\in L^1(\mathbb{R}^n,\mathrm{d}\bphi)$, $\mathcal{F}_{\bphi}w_0,\mathcal{F}_{\bphi}w_1\in L^1(\mathbb{R}^n)$, $1<\beta_0<2$ and $0<\alpha<2$. Notice now that by \cite[Example 5.12]{fractional-laplacian}, the solution of equation \eqref{example1} is given by 
\begin{align*}
&w(\bx,t)=w_0(\bx)+(\psi(t)-\psi(0))w_1(\bx) \\
&-\int_{\mathbb{R}^n}w_0\big(\bphi^{-1}\left(\bphi(\bx)-\bphi(\by)\right)\big)\times \\
&\hspace{2cm}\times\mathcal{F}_{\bphi}^{-1}\Big(|\bk|^{\alpha}I_{0+}^{\beta_0}\big[t^{\beta_0}E^{\beta_0}_{1,2\beta_0,\beta_0}\big(-|\bk|^{\alpha}t^{2\beta_0}\big)\big]\Big)(\by)\big|J\bphi(\by)\big|\,\mathrm{d}\by \\
&-\int_{\mathbb{R}^n}w_1\big(\bphi^{-1}\left(\bphi(\bx)-\bphi(\by)\right)\big)\times \\
&\hspace{2cm}\times\Big(\mathcal{F}_{\bphi}^{-1}\Big(|\bk|^{\alpha}I_{0+}^{\beta_0}\big[t^{\beta_0+1}E^{\beta_0}_{1,2\beta_0,\beta_0+1}\big(-|\bk|^{\alpha}t^{2\beta_0}\big)\big]\Big)(\by)\big|J\bphi(\by)\big|\,\mathrm{d}\by,
\end{align*}
where $E^{\lambda}_{\alpha,\beta,\gamma}(z)=\sum_{k=0}^{+\infty}c_k z^k$, $\gamma, z\in\mathbb{C},$ $\alpha,\beta\in\mathbb{R}$, $\lambda\in\mathbb{R}^+$ with
\[
c_0=1,\quad c_k=\prod_{j=0}^{k-1}\frac{\Gamma(\alpha[j\beta+\gamma]+1)}{\Gamma(\alpha[j\beta+\gamma]+\lambda+1)},\quad k=1,2,\ldots,
\]
is a generalised special function of Fox--Wright type. This is an entire function since the radius of convergence is $+\infty$. The proof of it can be done in the same way as in \cite[Lemma 5.2]{mittagbook}. Note that in the case $\lambda=\alpha$, the function $E^{\alpha}_{\alpha,\beta,\gamma}$ becomes the generalized (Kilbas--Saigo) Mittag-Leffler type function \cite[Chapter 5]{mittagbook}. Therefore, the solution of equation \eqref{phiexamplewave} can be calculated explicitly by using the above estimates.

\medskip If we consider the simpler example of heat type equation: 
\begin{equation}\label{example2}
\left\{
\begin{split}
\prescript{C}{}\partial_{t}^{\beta_0}w(\bx,t)&+t^{\beta_0}(-\Delta)^{\alpha/2}_{\bphi(\bx)} w(\bx,t)=0, \\
w(\bx,0)&=w_0(\bx),
\end{split}\right.
\end{equation}
where $w_0\in L^1(\mathbb{R}^n,\mathrm{d}\bphi)$, $\mathcal{F}_{\bphi}w_0\in L^1(\mathbb{R}^n)$, $0<\beta_0<1$ and $0<\alpha<2$. Again, it is easy to solve the heat analog of equation \eqref{phiexamplewave} ($0<\beta_0<1$ and $0<\alpha<2$)  by using Theorem \ref{cliffordnew}, and the same function $E^{\lambda}_{\alpha,\beta,\gamma}$ as in the previous example, knowing that the solution of problem \eqref{example2} is given by:
\begin{align*}
w(\bx,t)&=w_0(\bx)+\int_{\mathbb{R}^n}w_0\big(\bphi^{-1}\left(\bphi(\bx)-\bphi(\by)\right)\big)\times \\
&\hspace{1cm}\times\Big(\mathcal{F}_{\bphi}^{-1}\Big(|\bk|^{\alpha}I_{0+}^{\beta_0}\big[t^{\beta_0}E^{\beta_0}_{1,2\beta_0,\beta_0}\big(-|\bk|^{\alpha}t^{2\beta_0}\big)\big]\Big)(\by)\big|J\bphi(\by)\big|\,\mathrm{d}\by.  
\end{align*}

\section{Special cases of Dirac type operators}\label{special}

In this section, we show some relevant Dirac type operators as special case of the one introduced in formula \eqref{classdiracnew}. We also denote $I_{0+}^{\beta}$, $^{C}D_{0+}^{\beta}$ instead of $I_{0+}^{\beta,\psi}$, $^{C}D_{0+}^{\beta,\psi}$ when $\psi(t)\equiv t$. Some of the following examples can be found in \cite{BRS} but we include here for the sake of completeness. 

\subsection{ Wave Dirac type operator}
We begin with the following wave Dirac type operator:
\[ _{\bx,t}D_{t^{\alpha_0};\,t}^{1,\alpha_0}:=t^{\alpha_0/2}\,D_{\bx}+\mathfrak{f}\left(^{C}\partial_{t}^{\alpha_0}\right)+\mathfrak{f}^+,\] 
where $1<\alpha_0<2$ and $D_{\bx}=\sum_{k=1}^n e_k \partial_{\bx_k}$ is the Dirac operator, which factorizes the Laplacian as $D_{\bx}^2=-\Delta=-\sum_{k=1}^n \partial_{\bx_k}^2.$ We have that
\[(_{\bx,t}D_{t^{\alpha_0};\,t}^{1,\alpha_0})^2=-t^{\alpha_0}\Delta_{\bx}+\,^{C}\partial_{t}^{\alpha_0}.\]
Let us now recall the following fractional initial value problem
\begin{equation}\label{example1*}
\left\{
\begin{split}
^{C}\partial_{t}^{\alpha_0}w(\bx,t)-t^{\alpha_0}\Delta_{\bx} w(\bx,t)&=0, \\
w(\bx,t)|_{_{_{t=0+}}}&=w_0(\bx), \\
\partial_t w(\bx,t)|_{_{_{t=0+}}}&=w_1(\bx),
\end{split}\right.
\end{equation}
where $1<\alpha_0<2$. It was shown in Section \ref{examples} (see also  \cite[Section 3.2]{generalizedCauchyRS}) that the solution is given by
\begin{align}\label{ex1}
w(\bx,t)&=w_0(\bx)+w_1(\bx)t \nonumber\\
&-\int_{\mathbb{R}^n}\mathcal{F}^{-1}_{\bs}\big(I_{0+}^{\alpha_0}\big(|\bs|^{2}t^{\alpha_0}E^{\alpha_0}_{1,2\alpha_0,\alpha_0}(-|\bs|^{2}t^{2\alpha_0})\big)\big)(\bx-\by)w_0(\by){\rm d}\by\nonumber \\
&-\int_{\mathbb{R}^n}\mathcal{F}^{-1}_{\bs}\big(I_{0+}^{\alpha_0}\big(|\bs|^{2}t^{\alpha_0+1}E^{\beta_0}_{1,2\alpha_0,\alpha_0+1}(-|\bs|^{2}t^{2\alpha_0})\big)(\bx-\by)w_1(\by){\rm d}\by.
\end{align}

By using the above solution we can establish the following result. 

\begin{cor}
Let $1<\alpha_0<2$. Then the fractional Cauchy problem of wave type
\begin{equation}\label{ex2clifford}
\left\{ \begin{split}
\left(t^{\alpha_0/2}D_{\bx}+\mathfrak{f}\left(^{C}\partial_{t}^{\alpha_0}\right)+\mathfrak{f}^+\right)v(\bx,t)&=0,\,\, \bx\in\mathbb{R}^n,\,\,t\in (0,T],\\
v(\bx,t)|_{_{_{t=0}}}&=h_0(\bx), \\
\partial_t v(\bx,t)|_{_{_{t=0}}}&=h_1(\bx),\\
\end{split}
\right.
\end{equation}
can be solved, and the solution is given by
\begin{align*}
&v(\bx,t)=t^{\alpha_0/2}D_{\bx}(h_0(\bx))+t^{1+\alpha_0/2}D_{\bx}(h_1(\bx))\\ 
&+\frac{t^{\alpha_0/2}}{\Gamma(\alpha_0)}\int_{\mathbb{R}^n}h_0(\by)\left(\int_0^t (t-u)^{\alpha_0-1}u^{\alpha_0} \frac{(\bx-\by)}{(2\pi|\bx-\by|)^{n/2}}\times \right. \\
&\hspace{2cm}\left.\times\int_0^{+\infty}r^{2}E^{\alpha_0}_{1,2\alpha_0,\alpha_0}(-r^{2}u^{2\alpha_0})r^{n/2}rJ_{\frac{n}{2}}(r|\bx-\by|){\rm d}r{\rm d}u\right){\rm d}\by \\
&+\frac{t^{\alpha_0/2}}{\Gamma(\alpha_0)}\int_{\mathbb{R}^n}h_1(\by)\left(\int_0^t (t-u)^{\alpha_0-1}u^{\alpha_0+1} \frac{(\bx-\by)}{(2\pi|\bx-\by|)^{n/2}}\times\right. \\
&\hspace{2cm}\times\left.\int_0^{+\infty}r^{2}E^{\alpha_0}_{1,2\alpha_0,\alpha_0+1}(-r^{2}u^{2\alpha_0})r^{n/2}rJ_{\frac{n}{2}}(r|\bx-\by|){\rm d}r{\rm d}u\right){\rm d}\by \\
&+\mathfrak{f}\left(-t^{\alpha_0}\int_{\mathbb{R}^n}h_0(\by)\left(\frac{|\bx-\by|^{1-n/2}}{(2\pi)^{n/2}}\times \right.\right.\\
&\hspace{2cm}\left.\left.\times\int_0^{+\infty}r^{2}E^{\alpha_0}_{1,2\alpha_0,\alpha_0}(-r^{2}t^{2\alpha_0})r^{n/2}J_{\frac{n}{2}-1}(r|\bx-\by|){\rm d}r\right){\rm d}y\right. \\
&-\int_{\mathbb{R}^n}h_1(\by)\left(\frac{|\bx-\by|^{1-n/2}}{(2\pi)^{n/2}}\times\right. \\
&\hspace{2cm}\times\left.\left.\int_0^{+\infty}r^{2}E^{\alpha_0}_{1,2\alpha_0,\alpha_0+1}(-r^{2}t^{2\alpha_0})r^{n/2}J_{\frac{n}{2}-1}(r|\bx-\by|){\rm d}r\right){\rm d}\by\right) \\
&\mathfrak{f}^{+}\bigg(h_0(\bx)+h_1(\bx)t\bigg. \\
&-\frac{1}{\Gamma(\alpha_0)}\int_{\mathbb{R}^n}h_0(\by)\left(\int_0^t (t-u)^{\alpha_0-1}u^{\alpha_0} \frac{|\bx-\by|^{1-n/2}}{(2\pi)^{n/2}}\times \right. \\
&\hspace{2cm}\left. \times\int_0^{+\infty}r^{2}E^{\alpha_0}_{1,2\alpha_0,\alpha_0}(-r^{2}u^{2\alpha_0})r^{n/2}J_{\frac{n}{2}-1}(r|\bx-\by|){\rm d}r{\rm d}u\right){\rm d}\by \\
&-\frac{1}{\Gamma(\alpha_0)}\int_{\mathbb{R}^n}h_1(\by)\left(\int_0^t (t-u)^{\alpha_0-1}u^{\alpha_0+1} \frac{|\bx-\by|^{1-n/2}}{(2\pi)^{n/2}}\times\right. \\
&\hspace{2cm}\times\left.\left.\int_0^{+\infty}r^{2}E^{\alpha_0}_{1,2\alpha_0,\alpha_0+1}(-r^{2}u^{2\alpha_0})r^{n/2}J_{\frac{n}{2}-1}(r|\bx-\by|){\rm d}r{\rm d}u\right){\rm d}\by\right).
\end{align*}
\end{cor}
\begin{proof}
By Theorem \ref{cliffordnew} we have that the solution of equation \eqref{ex2clifford} is given by the application of $\big(t^{\beta_0/2}\,D_{\bx}+\mathfrak{f}\left(^{C}\partial_{t}^{\beta_0}\right)+\mathfrak{f}^+\big)$ to the representation \eqref{ex1}. Let us then calculate each component of the solution. First we need to recall some useful estimates. Note that formula (25.11) in \cite[Lemma 25.1]{samko} implies that 
\begin{equation}\label{radial}
\frac{1}{(2\pi)^n}\int_{\mathbb{R}^n}e^{-i\bs\cdot \bx}\varphi(|\bs|){\rm d}\bs=\frac{|\bx|^{1-n/2}}{(2\pi)^{n/2}}\int_0^{+\infty}\varphi(r)r^{n/2}J_{\frac{n}{2}-1}(r|\bx|){\rm d}r,
\end{equation}
where $J_\nu$ denotes the Bessel function with index $\nu$ (for more details see e.g.  \cite{new1}) and $\varphi$ is a radial function such that 
\[\int_0^{+\infty}\tau^{n-1}(1+\tau)^{(1-n)/2}|\varphi(\tau)|{\rm d}\tau<+\infty,\]
provided that the integral on the left-hand side of \eqref{radial} is interpreted as conventionally convergent. It converges absolutely if 
\[\int_0^{+\infty}\tau^{n-1}|\varphi(\tau)|{\rm d}\tau<+\infty.\]
And, we also have that \cite{ferreira2017mmas}:
\begin{equation}\label{ferreira9}
D_{\bx}\left(|\bx|^{1-\frac{n}{2}}J_{\frac{n}{2}-1}(r|\bx|)\right)=-\frac{r\bx}{|\bx|^{\frac{n}{2}}}J_{\frac{n}{2}}(r|\bx|),\quad \bx\in\mathbb{R}^n,\,\,r\geqslant0.
\end{equation}
By \eqref{ex1} and \eqref{radial} we get
\begin{align}
&w(\bx,t)=h_0(\bx)+h_1(\bx)t\nonumber \\
&-\frac{1}{\Gamma(\alpha_0)}\int_{\mathbb{R}^n}h_0(\by)\left(\int_0^t (t-u)^{\alpha_0-1}u^{\alpha_0} \frac{|\bx-\by|^{1-n/2}}{(2\pi)^{n/2}}\times \right. \nonumber \\
&\hspace{2cm}\left. \times\int_0^{+\infty}r^{2}E^{\alpha_0}_{1,2\alpha_0,\alpha_0}(-r^{2}u^{2\alpha_0})r^{n/2}J_{\frac{n}{2}-1}(r|\bx-\by|){\rm d}r{\rm d}u\right){\rm d}\by\nonumber \\
&-\frac{1}{\Gamma(\alpha_0)}\int_{\mathbb{R}^n}h_1(\by)\left(\int_0^t (t-u)^{\alpha_0-1}u^{\alpha_0+1} \frac{|\bx-\by|^{1-n/2}}{(2\pi)^{n/2}}\times\right.\nonumber \\
&\hspace{2cm}\times\left.\int_0^{+\infty}r^{2}E^{\alpha_0}_{1,2\alpha_0,\alpha_0+1}(-r^{2}u^{2\alpha_0})r^{n/2}J_{\frac{n}{2}-1}(r|\bx-\by|){\rm d}r{\rm d}u\right){\rm d}\by.\label{solutionexample}
\end{align}
Let us calculate each component of $_{\bx,t}D_{t^{\alpha_0};\,t}^{1,\alpha_0}w(\bx,t)$ where $w(\bx,t)$ is given in  \eqref{solutionexample}. 

By \eqref{ferreira9}, we get the first component:
\begin{align*}
t^{\alpha_0/2}&\,D_{\bx} w(\bx,t)=t^{\alpha_0/2}D_{\bx}(h_0(\bx))+t^{1+\alpha_0/2}D_{\bx}(h_1(\bx)) \\
&+\frac{t^{\alpha_0/2}}{\Gamma(\alpha_0)}\int_{\mathbb{R}^n}h_0(\by)\left(\int_0^t (t-u)^{\alpha_0-1}u^{\alpha_0} \frac{(\bx-\by)}{(2\pi|\bx-\by|)^{n/2}}\times \right. \\
&\hspace{1.5cm}\left. \times\int_0^{+\infty}r^{2}E^{\alpha_0}_{1,2\alpha_0,\alpha_0}(-r^{2}u^{2\alpha_0})r^{n/2}rJ_{\frac{n}{2}}(r|\bx-\by|){\rm d}r{\rm d}u\right){\rm d}\by \\
&+\frac{t^{\alpha_0/2}}{\Gamma(\alpha_0)}\int_{\mathbb{R}^n}h_1(\by)\left(\int_0^t (t-u)^{\alpha_0-1}u^{\alpha_0+1} \frac{(\bx-\by)}{(2\pi|\bx-\by|)^{n/2}}\times\right. \\
&\hspace{1.5cm}\times\left.\int_0^{+\infty}r^{2}E^{\alpha_0}_{1,2\alpha_0,\alpha_0+1}(-r^{2}u^{2\alpha_0})r^{n/2}rJ_{\frac{n}{2}}(r|\bx-\by|){\rm d}r {\rm d}u\right){\rm d}\by.
\end{align*}
By \eqref{solutionexample} and Theorem \ref{dirl} we obtain the second component as follows:
\begin{align*}
&^{C}\partial_t^{\alpha_0}w(\bx,t) \\
&=-t^{\alpha_0}\int_{\mathbb{R}^n}h_0(\by)\left(\frac{|\bx-\by|^{1-n/2}}{(2\pi)^{n/2}}\times \right. \\
&\hspace{2cm}\left. \times\int_0^{+\infty}r^{2}E^{\alpha_0}_{1,2\alpha_0,\alpha_0}(-r^{2}t^{2\alpha_0})r^{n/2}J_{\frac{n}{2}-1}(r|\bx-\by|){\rm d}r\right){\rm d}\by \\
&-t^{\alpha_0+1}\int_{\mathbb{R}^n}h_1(\by)\left(\frac{|\bx-\by|^{1-n/2}}{(2\pi)^{n/2}}\times\right. \\
&\hspace{2cm}\times\left.\int_0^{+\infty}r^{2}E^{\alpha_0}_{1,2\alpha_0,\alpha_0+1}(-r^{2}t^{2\alpha_0})r^{n/2}J_{\frac{n}{2}-1}(r|\bx-\by|){\rm d}r\right){\rm d}\by,
\end{align*}
and this completes the proof.
\end{proof}

Similarly as the above statement the following assertion can be obtained.  
\begin{cor}
Let $0<\alpha_0<1$. Then the fractional Cauchy problem of heat type 
\begin{equation}
\left\{ \begin{split}
\left(t^{\alpha_0/2}D_{\bx}+\mathfrak{f}\left(^{C}\partial_{t}^{\alpha_0}\right)+\mathfrak{f}^+\right)v(\bx,t)&=0,\,\, \bx\in\mathbb{R}^n,\,\,t\in (0,T],\\
v(\bx,t)|_{_{_{t=0}}}&=h_0(\bx),  \\
\end{split}
\right.
\end{equation}
can be solved, and the solution is given by
\begin{align*}
&v(\bx,t)=t^{\alpha_0/2}D_{\bx}(h_0(\bx))+\frac{t^{\alpha_0/2}}{\Gamma(\alpha_0)}\int_{\mathbb{R}^n}h_0(\by)\left(\int_0^t (t-u)^{\alpha_0-1}u^{\alpha_0} \times \right. \\
&\hspace{0.5cm}\left.\times\frac{(\bx-\by)}{(2\pi|\bx-\by|)^{n/2}}\int_0^{+\infty}r^{2}E^{\alpha_0}_{1,2\alpha_0,\alpha_0}(-r^{2}u^{2\alpha_0})r^{n/2}rJ_{\frac{n}{2}}(r|\bx-\by|){\rm d}r {\rm d}u\right){\rm d}\by \\
&\hspace{1cm}+\mathfrak{f}\left(-t^{\alpha_0}\int_{\mathbb{R}^n}h_0(\by)\left(\frac{|\bx-\by|^{1-n/2}}{(2\pi)^{n/2}}\times \right.\right.\\
&\hspace{2cm}\left.\left.\left.\times\int_0^{+\infty}r^{2}E^{\alpha_0}_{1,2\alpha_0,\alpha_0}(-r^{2}t^{2\alpha_0})r^{n/2}J_{\frac{n}{2}-1}(r|\bx-\by|){\rm d}r\right){\rm d}\by\right.\right) \\
&\hspace{1cm}+\mathfrak{f}^{+}\left(h_0(\bx)-\frac{1}{\Gamma(\alpha_0)}\int_{\mathbb{R}^n}h_0(\by)\left(\int_0^t (t-u)^{\alpha_0-1}u^{\alpha_0} \frac{|\bx-\by|^{1-n/2}}{(2\pi)^{n/2}}\times \right.\right. \\
&\hspace{2cm}\left.\left.\times\int_0^{+\infty}r^{2}E^{\alpha_0}_{1,2\alpha_0,\alpha_0}(-r^{2}u^{2\alpha_0})r^{n/2}J_{\frac{n}{2}-1}(r|\bx-\by|){\rm d}r {\rm d}u\right){\rm d}\by\right). 
\end{align*}
\end{cor}
\subsection{ Fractional telegraph Dirac operator} Now we consider the following fractional telegraph Dirac operator
\[ _{\bx,t}D_{a,c;\,t}^{1,\alpha_0,\alpha_1}:=c\,D_{\bx}+\mathfrak{f}\left(^{C}\partial_{t}^{\alpha_0}+a ^{C}\partial_{t}^{\alpha_1}\right)+\mathfrak{f}^+,\] 
where $D_{\bx}$ is the Dirac operator, $a\geqslant0$, $c>0$, $0<\alpha_1\leqslant1$ and $1<\alpha_0<2$. We have 
\[(_{\bx,t}D_{a,c;\,t}^{1,\alpha_0,\alpha_1})^2=-c^2\,\Delta_{\bx}+\,^{C}\partial_{t}^{\alpha_0}+a\,^{C}\partial_{t}^{\alpha_1}.\]

Here we consider the case of constant coefficients. By using the obtained results we can directly prove the following statement. Nevertheless, we omit all calculations since it is also an analogue of \cite[Theorem 4.1]{ferreira2018}.    

\begin{cor}
Let $1<\alpha_0<2$, $0<\alpha_1\leqslant 1$, $a\geqslant0$ and $c>0$. Then the fractional Cauchy type problem 
\begin{equation}\label{constante}
\left\{ \begin{split}
\left(cD_{\bx}+\mathfrak{f}\left(^{C}\partial_{t}^{\alpha_0}+a\,^{C}\partial_{t}^{\alpha_1}\right)+\mathfrak{f}^+\right)v(\bx,t)&=0,\quad \bx\in\mathbb{R}^n,\,\,t\in (0,T],\\
v(\bx,t)|_{_{_{t=0}}}&=h_0(\bx), \\
\partial_t v(\bx,t)|_{_{_{t=0}}}&=h_1(\bx),\\
\end{split}
\right.
\end{equation}
can be solved, and the solution is given by
\begin{align*}
v(\bx,t)=\int_{\mathbb{R}^n}\mathfrak{H}_0^{\alpha_0,\alpha_1}(\bx-\by,t)h_0(\by){\rm d}\by+\int_{\mathbb{R}^n}\mathfrak{H}_1^{\alpha_0,\alpha_1}(\bx-\by,t)h_1(\by){\rm d}\by, 
\end{align*}
where $\mathfrak{H}_0^{\alpha_0,\alpha_1}$ and $\mathfrak{H}_1^{\alpha_0,\alpha_1}$ are the first and second fundamental solutions given in formulas (4.3) and (4.4) of \cite{ferreira2018}.\end{cor}

For the particular case $a=0$, the above result coincides with the result in \cite{ferreira2017} for the time-fractional parabolic Dirac operator.

\section{Inverse problems}\label{directpro}

Now we combine some results given in Section \ref{mainresults} with a new method to recover the variable coefficient of an inverse Cauchy type problem by the consideration of two (direct) fractional Cauchy type problems. The method was introduced recently in \cite{su20}, and extended to some fractional differential equations in \cite{generalizedCauchyRS} as well. In this section, we extend some recent results from \cite{BRS} by using the Riemann-Liouville fractional derivative of complex order, with respect to another function. Here we just consider the fractional Laplacian in the classical sense since we are able to use some well-known properties of this operator which could not hold for the case of the fractional Laplacian with respect to a function. Some recent results of \cite{generalizedCauchyRS,RRS} are used to establish the newer statements. We mainly focus in solving some inverse fractional Cauchy problems of wave and heat type.

\subsection{ Fractional wave type equations} We will study how to recover the variable coefficient $\rho(t)$ for the following fractional Cauchy problem of wave type:
\begin{equation}\label{eq1Cauchylast}
\left\{ \begin{split}
\big(\rho^{1/2}(t)D_{\bx}+\mathfrak{f}\big(\,^{C}\partial_t^{\alpha,\psi}\big)+\mathfrak{f}^{+}\big)w(\bx,t)&=0, 
\qquad \bx \in \mathbb R^{n}, \,\, 0<t\leqslant T<+\infty,\\
w(\bx,t)|_{_{_{t=0}}}&=w_0(\bx), \\
\partial_t w(\bx,t)|_{_{_{t=0}}}&=w_1(\bx),
\end{split}
\right.
\end{equation}
where $1<\alpha<2$, $\rho(t)$ is a continuous function and $w_{0,1}$ are to be satisfied certain hypothesis. As usual, we denote  
\begin{align*}
^{C}\partial_t^{\alpha,\psi}w(\bx,t)&:=\,^{C}D_{0+}^{\alpha,\psi}w(\bx,t) \\
&=D_{0+}^{\alpha,\psi}\big(w(\bx,t)-w(\bx,t)|_{_{_{t=0}}}-\partial_t w(\bx,t)|_{_{_{t=0}}}(\psi(t)-\psi(0))\big).
\end{align*}
Notice that the formal passage $\alpha\to1$ transforms the equation \eqref{eq1Cauchylast} to the heat type, while \eqref{eq1Cauchylast} transforms into a wave type equation in the particular case $\alpha=2$. Let us now remember (see formula \ref{factornew}) that applying $\big(\rho^{1/2}(t)D_{\bx}+\mathfrak{f}\big(\,^{C}\partial_t^{\alpha}\big)+\mathfrak{f}^{+}\big)$ to equation \eqref{eq1Cauchylast}, it follows that 
\begin{equation}\label{eq1Cauchylast1}
\left\{\begin{split}
\big(-\rho(t)\Delta_{\bx}+\,^{C}\partial_t^{\alpha,\psi}\big)w(\bx,t)&=0, 
\qquad \bx\in \mathbb R^{n}, \quad 0<t\leqslant T<+\infty,\\
w(\bx,t)|_{_{_{t=0}}}&=w_0(\bx), \\
\partial_t w(\bx,t)|_{_{_{t=0}}}&=w_1(\bx),
\end{split}
\right.
\end{equation}
whose solution is given by (Theorem \ref{cor3.2.1CP})
\begin{align}
w(\bx,t)-w_0(\bx)-&w_1(\bx)(\psi(t)-\psi(0))=\int_{\mathbb{R}^n}\big(\mathcal{F}^{-1}_{\bs}\mathcal{H}_1(t,|\bs|^{2},\rho\big)(\bx-\by)w_1(\by){\rm d}\by \nonumber\\
&+\int_{\mathbb{R}^n}\big(\mathcal{F}^{-1}_{\bs}\mathcal{H}_0(t,|\bs|^{2},\rho\big)(\bx-\by)w_0(\by){\rm d}\by.\label{transsolution}
\end{align}
We then study the following two direct fractional Cauchy type problems with the aim of recovering the continuous (in time) variable coefficient $\rho$: 
\begin{equation}\label{eq1invlast}
\left\{ \begin{split}
\big(\rho^{1/2}(t)D_{\bx}+\mathfrak{f}\big(\,^{C}\partial_t^{\alpha,\psi}\big)+\mathfrak{f}^{+}\big)w(\bx,t)&=0, 
\qquad \bx\in\mathbb{R}^{n}, \quad 0<t\leqslant T<+\infty,\\
w(\bx,t)|_{_{_{t=0}}}&=0, \\
\partial_t w(\bx,t)|_{_{_{t=0}}}&=w_1(\bx), 
\end{split}
\right.
\end{equation}
and 
\begin{equation}\label{eq2invlast}
\left\{ \begin{split}
\big(\rho^{1/2}(t)D_{\bx}+\mathfrak{f}\big(\,^{C}\partial_t^{\alpha,\psi}\big)+\mathfrak{f}^{+}\big)v(\bx,t)&=0, 
\qquad \bx \in \mathbb{R}^{n}, \quad 0<t\leqslant T<+\infty,\\
v(\bx,t)|_{_{_{t=0}}}&=0, \\
\partial_t v(\bx,t)|_{_{_{t=0}}}&=\Delta_{\bx} w_1(\bx),
\end{split}
\right.
\end{equation}
where $\Omega\subset\mathbb{R}^n$ is an open, bounded set with piecewise smooth boundary $\partial\Omega$ and $w_1\in C^2 (\Omega)\cap C^1 (\overline{\Omega})$ is a given function which satisfy $\text{supp}(w_1)\subset\Omega$. To apply the method showed in \cite[Section 5]{generalizedCauchyRS}, we need an additional initial data at a fixed point $\bq\in\Omega$. Indeed, let us fix an arbitrary observation point $\bq\in\Omega$ for two time dependent values:
\[h_1(t):=w(\bx,t)|_{_{_{\bx=\bq}}}-w_1(\bx)|_{_{_{\bx=\bq}}}(\psi(t)-\psi(0)),\quad h_2(t):=v(\bx,t)|_{_{_{\bx=\bq}}},\,\, 0<t\leqslant T,\] 
where $w(\bx,t)$ and $v(\bx,t)$ are the solutions (at the considered time) of the transformed  equations \eqref{eq1invlast} and \eqref{eq2invlast} after the application of $\big(\rho^{1/2}(t)D_{\bx}+\mathfrak{f}\big(\,^{C}\partial_t^{\alpha,\psi}\big)+\mathfrak{f}^{+}\big)$ respectively. In the proof of the next result will be very clear where those quantities play their role. 

\medskip Let us now establish one of the main results in this section.

\begin{thm}\label{inversewavelast}
Let the following conditions be satisfied:
 \begin{enumerate}
 \item $w_1\in C^2 (\Omega)\cap C^1(\overline{\Omega})$ and $\text{supp}(w_1)\subset\Omega$,
 \item $h_2\in C^{2}[0,T]$ and $h_2(t)\neq0$ for any $t\in(0,T)$,
 \item $\frac{D_{0+}^{\alpha,\psi}h_1(t)}{h_2(t)}\geqslant K>0$ for any $t\in(0,T]$. 
 \end{enumerate}
 Then, the fractional inverse Cauchy problem \eqref{eq1invlast}, \eqref{eq2invlast} has a solution given by
 \[\rho(t)=\frac{D_{0+}^{\alpha,\psi}h_1(t)}{h_2(t)},\quad t\in(0,T].\]
\end{thm}
\begin{proof}
By Theorem \ref{cor3.2.1CP}, the solutions of equations \eqref{eq1invlast} and \eqref{eq2invlast} are given by the application of $\big(\rho^{1/2}(t)D_{\bx}+\mathfrak{f}\big(\,^{C}\partial_t^{\alpha}\big)+\mathfrak{f}^{+}\big)$ to the following representations respectively:
\[w(\bx,t)-w_1(\bx)(\psi(t)-\psi(0))=\int_{\mathbb{R}^n}\mathcal{F}^{-1}_{\bs}\big(\mathcal{H}_1(t,|\bs|^{2},\rho\big)(\bx-\by)\big(\chi_{\Omega}w_1\big)(\by){\rm d}\by,\]
and
\[v(\bx,t)-\Delta_{\bx} w_1(\bx)(\psi(t)-\psi(0))=\hspace{-0.2cm}\int_{\mathbb{R}^n}\mathcal{F}^{-1}_{\bs}\big(\mathcal{H}_1(t,|\bs|^{2},\rho\big)(\bx-\by)\big(\chi_\Omega \Delta_{\by} w_1\big)(\by){\rm d}\by,\]
where $\chi_{\Omega}$ is the characteristic function of $\Omega$. Due to the additional data at $\bq\in\Omega$,
\begin{align*}
h_1(t)=\int_{\Omega}\mathcal{F}^{-1}_{\bs}\big(\mathcal{H}_1(t,|\bs|^2,\rho)\big)(\bx-\by)|_{_{\bx=\bq}}w_1(\by){\rm d}\by,
\end{align*}  
and 
\begin{align*}
h_2(t)=\Delta_{\bx} &w_1(\bx)|_{_{_{\bx=q}}}(\psi(t)-\psi(0)) \\
&+\int_{\Omega}\mathcal{F}^{-1}_{\bs}\big(\mathcal{H}_1(t,|\bs|^2,\rho)\big)(\bx-\by)|_{_{\bx=\bq}}\Delta_{\by} w_1(\by){\rm d}\by.
\end{align*}  
By \eqref{eq1Cauchylast1}, \eqref{transsolution}, the definition of $h_1$ and $h_2$ at the point $\bx=\bq$, and by the Green second formula we arrive at  
\begin{align*}
&D^{\alpha,\psi}_{0+}h_1(t)=\,^{C}\partial_t^{\alpha,\psi}w(\bx,t)|_{_{_{\bx=\bq}}} \\
&=\rho(t)\Delta_{\bx}\left(\int_{\Omega}\mathcal{F}^{-1}_{\bs}\big(\mathcal{H}_1(t,|\bs|^2,\rho)\big)(\bx-\by)|_{_{\bx=\bq}}w_1(\by){\rm d}\by \right.\\
&\hspace{7cm}\left.+w_1(\bx)|_{_{_{\bx=\bq}}}(\psi(t)-\psi(0))\right) \\
&=\rho(t)\bigg(\int_{\Omega}\Delta_{\by} \left(\mathcal{F}^{-1}_{\bs}\big(\mathcal{H}_1(t,|\bs|^2,\rho)\big)(\bx-\by)|_{_{\bx=\bq}}\right)w_1(\by){\rm d}\by \bigg. \\
&\hspace{7cm}\bigg.+\Delta_{\bx} w_1(\bx)|_{_{_{\bx=\bq}}}(\psi(t)-\psi(0))\bigg) \\
&=\rho(t)\left(\int_{\Omega}\mathcal{F}^{-1}_{\bs}\big(\mathcal{H}_1(t,|\bs|^2,\rho)\big)(\bx-\by)|_{_{\bx=\bq}}\Delta_{\by}w_1(\by){\rm d}\by\right. \\
&\hspace{2cm}\left.+\Delta_{\bx} w_1(\bx)|_{_{_{\bx=\bq}}}(\psi(t)-\psi(0))\right)=\rho(t)\big(v(\bx,t)|_{_{_{\bx=\bq}}}\big)=\rho(t)h_2(t).
\end{align*}    
Notice that we are abusing of the notation above by evaluating at the point $\bx=\bq$. Indeed, we first need to do the calculations and in the last step, we just evaluate at the point $\bx=\bq.$ 

On the other hand, the function $\rho$ is assumed to be continuous, therefore we require that $h_2(t)\neq0$ for any $t\in(0,T]$. Note also that we would like to get a non-trivial function and positive, hence $\rho(t)\geqslant K>0$. Thus, we have to request at the beginning $\frac{D^{\alpha}_{0+}h_1(t)}{h_2(t)}\geqslant K>0$ for any $t\in(0,T]$.   
\end{proof}

\subsection{ Fractional heat type equations} Let us study the inverse problem in recovering the thermal diffusivity $\rho$ in the following fractional Cauchy problem
\begin{equation}\label{eq1CauchyParaboliclast}
\left\{ \begin{split}
\big(\rho^{1/2}(t)D_{\bx}+\mathfrak{f}\big(\,^{C}\partial_t^{\alpha,\psi}\big)+\mathfrak{f}^{+}\big)w(\bx,t)&=0, 
\qquad \bx \in \mathbb R^{n}, \quad 0<t\leqslant T<+\infty,\\
w(\bx,t)|_{_{_{t=0}}}&=w_0(\bx), 
\end{split}
\right.
\end{equation}
where $0<\alpha<1$ and $\rho(t)>0$ is assumed to be a continuous function. In this case 
\[^{C}\partial_t^{\alpha,\psi}w(\bx,t)=\,^{C}D_{0+}^{\alpha,\psi}w(\bx,t)=D_{0+}^{\alpha,\psi}\big(w(\bx,t)-w(\bx,t)|_{_{_{t=0}}}\big),\quad \bx\in\mathbb{R}^n, \,\,t>0,\]
and we get the heat type equation in the particular case $\alpha=1$ in \eqref{eq1CauchyParaboliclast}. 

To reconstruct the variable coefficient in \eqref{eq1CauchyParaboliclast}, we use a similar procedure to the case of the fractional wave equation. As before, we suppose that $\Omega\subset\mathbb{R}^n$ is an open, bounded set with a piecewise smooth boundary $\partial\Omega$ and $\bq\in\Omega$ is a fixed point. Thus, we solve the inverse problem \eqref{eq1CauchyParaboliclast} by studying two fractional Cauchy problems with additional data at the point $\bq\in\Omega$: 
\begin{equation}\label{eq1invParaboliclast}
\left\{ \begin{split}
\big(\rho^{1/2}(t)D_{\bx}+\mathfrak{f}\big(\,^{C}\partial_t^{\alpha,\psi}\big)+\mathfrak{f}^{+}\big)w(\bx,t)&=0, 
\qquad \bx \in \mathbb R^{n}, \quad 0<t\leqslant T<+\infty,\\
w(\bx,t)|_{_{_{t=0}}}&=w_1(\bx), \\
w(\bx,t)|_{_{_{\bx=\bq}}}-w_1(\bx)|_{_{_{\bx=\bq}}}&=h_1(t),
\end{split}
\right.
\end{equation}
and \begin{equation}\label{eq2invParaboliclast}
\left\{ \begin{split}
\big(\rho^{1/2}(t)D_{\bx}+\mathfrak{f}\big(\,^{C}\partial_t^{\alpha,\psi}\big)+\mathfrak{f}^{+}\big)v(\bx,t)&=0, 
\qquad \bx \in \mathbb{R}^{n}, \quad 0<t\leqslant T<+\infty,\\
v(\bx,t)|_{_{_{t=0}}}&=\Delta_{\bx} w_1(\bx),\\
v(\bx,t)|_{_{_{\bx=\bq}}}&=h_2(t),
\end{split}
\right.
\end{equation}
where $w_1\in C^2 (\Omega)\cap C^1(\overline{\Omega})$ and $\text{supp}(w_1)\subset\Omega \subset\mathbb{R}^n$. 

\begin{thm}\label{inverseheat}
Let the following conditions be satisfied:
 \begin{enumerate}
 \item $w_1\in C^1 (\Omega)\cap C^1(\overline{\Omega})$ and $\text{supp}(w_1)\subset\Omega$,
 \item $h_2\in C^{1}[0,T]$ such that $h_2(t)\neq0$ for any $t\in(0,T)$,
 \item $\frac{D_{0+}^{\alpha,\phi}h_1(t)}{h_2(t)}\geqslant K>0$ for any $t\in(0,T]$. 
 \end{enumerate}
Then, the fractional inverse Cauchy problem \eqref{eq1invParaboliclast}, \eqref{eq2invParaboliclast} has a solution given by
 \[\rho(t)=\frac{D_{0+}^{\alpha,\psi}h_1(t)}{h_2(t)},\quad t\in(0,T].\]
\end{thm}
\begin{proof}
The solutions of the fractional Cauchy problems \eqref{eq1invParaboliclast} and \eqref{eq2invParaboliclast} are given by the application of $\big(\rho^{1/2}(t)D_{\bx}+\mathfrak{f}\big(\,^{C}\partial_t^{\alpha,\psi}\big)+\mathfrak{f}^{+}\big)$ to the following representations respectively: 
\begin{align*}
w(\bx,t)-w_1(\bx)=\int_{\mathbb{R}^n}\mathcal{F}^{-1}_{\bs}\big(\mathcal{H}_0(t,|\bs|^2,\rho)\big)(\bx-\by)\big(\chi_{\Omega} w_1\big)(\by){\rm d}\by,
\end{align*}  
and 
\begin{align*}
v(\bx,t)-\Delta_{\bx} w_1(\bx)=\int_{\mathbb{R}^n}\mathcal{F}^{-1}_{\bs}\big(\mathcal{H}_0(t,|\bs|^2,\rho)\big)(\bx-\by)\big(\chi_{\Omega}\Delta_{\by}w_1\big)(\by){\rm d}\by.
\end{align*}
By the additional data at the point $\bq\in\Omega$, we also have
\begin{align*}
h_1(t)=\int_{\Omega}\mathcal{F}^{-1}_{\bs}\big(\mathcal{H}_0(t,|\bs|^2,\rho)\big)(\bx-\by)|_{_{\bx=\bq}}w_1(\by){\rm d}\by
\end{align*}  
and 
\begin{align*}
h_2(t)=\int_{\Omega}\mathcal{F}^{-1}_{\bs}\big(\mathcal{H}_0(t,|\bs|^2,\rho)\big)(\bx-\by)|_{_{\bx=\bq}}\Delta_{\by} w_1(\by){\rm d}\by+\Delta_{\bx} w_1(\bx)|_{_{_{\bx=\bq}}}.
\end{align*}  
By \eqref{eq1Cauchylast1}, \eqref{transsolution}, the definition of $h_1$ and $h_2$ at the point $\bx=\bq$, and by the Green second formula we get the following equivalences:
\begin{align*}
&D^{\alpha,\psi}_{0+}h_1(t)=\,^{C}\partial_t^{\alpha,\psi}w(\bx,t)|_{_{_{\bx=\bq}}} \\
&=\rho(t)\Delta_{\bx}\left(\int_{\Omega}\mathcal{F}^{-1}_{\bs}\big(\mathcal{H}_0(t,|\bs|^2,\rho)\big)(\bx-\by)|_{_{\bx=\bq}}w_1(\by){\rm d}\by+w_1(\bx)|_{_{_{\bx=\bq}}}\right) \\
&=\rho(t)\left(\int_{\Omega}\Delta_{\by} \left(\mathcal{F}^{-1}_{\bs}\big(\mathcal{H}_0(t,|\bs|^2,\rho)\big)(\bx-\by)|_{_{\bx=\bq}}\right)w_1(\by){\rm d}\by+\Delta_{\bx} w_1(\bx)|_{_{_{\bx=\bq}}}\right) \\
&=\rho(t)\left(\int_{\Omega}\mathcal{F}^{-1}_{\bs}\big(\mathcal{H}_0(t,|\bs|^2,\rho)\big)(\bx-\by)|_{_{\bx=\bq}}\Delta_{\by}w_1(\by){\rm d}\by+\Delta_{\bx} w_1(\bx)|_{_{_{\bx=\bq}}}\right) \\
&=\rho(t)\big(v(\bx,t)|_{_{_{\bx=\bq}}}\big)=\rho(t)h_2(t),
\end{align*}    
which finishes the proof.
\end{proof}

\subsection{ Examples} Let us consider the particular case $1<\beta_0<2$ and $\psi(t)=t$ of Theorem \ref{inversewavelast}. We have
\begin{equation}\label{eq1Cauchyinv1}
\left\{ \begin{split}
\big(t^{\beta_{0}/2}D_{\bx}+\mathfrak{f}\big(\,^{C}\partial_t^{\beta_0}\big)+\mathfrak{f}^{+}\big)w(\bx,t)&=0, 
\qquad \bx \in \mathbb{R}^{n}, \quad 0<t\leqslant T<+\infty,\\
w(\bx,t)|_{_{_{t=0}}}&=0, \\
\partial_t w(\bx,t)|_{_{_{t=0}}}&=w_1(\bx), \\
w(\bx,t)|_{_{_{\bx=\bq}}}-w_1(\bx)|_{_{_{\bx=\bq}}}t&=h_1(t),\quad \bq\in\Omega,
\end{split}
\right.
\end{equation}
and
\begin{equation}\label{eq1Cauchyinv2}
\left\{ \begin{split}
\big(t^{\beta_{0}/2}D_{\bx}+\mathfrak{f}\big(\,^{C}\partial_t^{\beta_0}\big)+\mathfrak{f}^{+}\big)v(\bx,t)&=0, 
\qquad \bx \in \mathbb R^{n}, \quad 0<t\leqslant T<+\infty,\\
v(\bx,t)|_{_{_{t=0}}}&=0, \\
\partial_t v(\bx,t)|_{_{_{t=0}}}&=\Delta_{\bx} w_1(\bx),\\
v(\bx,t)|_{_{_{\bx=\bq}}}&=h_2(t),\quad \bq\in\Omega, 
\end{split}
\right.
\end{equation}
where $\text{supp}(w_1)\subset\Omega\subset\mathbb{R}^n$. The solutions of \eqref{eq1Cauchyinv1} and \eqref{eq1Cauchyinv2} are given by the application of $\big(t^{\beta_{0}/2}D_{\bx}+\mathfrak{f}\big(\,^{C}\partial_t^{\beta_0}\big)+\mathfrak{f}^{+}\big)$ to the following representations respectively (see formula \eqref{ex1}):
\begin{align}\label{sosinex1}
w(\bx,t)&=w_1(\bx)t \nonumber\\ 
&-\int_{\mathbb{R}^n}\mathcal{F}^{-1}_{\bs}\big(I_{0+}^{\beta_0}\big(|\bs|^{2}t^{\beta_0+1}E^{\beta_0}_{1,2\beta_0,\beta_0+1}(-|\bs|^{2}t^{2\beta_0})\big)(\bx-\by)w_1(\by){\rm d}\by,
\end{align}
and 
\begin{align}\label{sosinex2}
v(\bx,t)&=\Delta_{\bx} w_1(\bx)t \nonumber\\
&-\int_{\mathbb{R}^n}\mathcal{F}^{-1}_{\bs}\big(I_{0+}^{\beta_0}\big(|\bs|^{2}t^{\beta_0+1}E^{\beta_0}_{1,2\beta_0,\beta_0+1}(-|\bs|^{2}t^{2\beta_0})\big)(\bx-\by)\Delta_{\by}w_1(\by){\rm d}\by.
\end{align}
Further, by Theorem \ref{dirl} and \eqref{sosinex1} we obtain
\[D_{0+}^{\beta_0}h_1(t)=t^{\beta_{0}+1}\int_{\Omega}w_1(\by)\left(\mathcal{F}^{-1}_{\bs}\big(|\bs|^{2}E_{1,2\beta_{0},\beta_{0}+1}^{\beta_0}(-|\bs|^{2}t^{2\beta_0})\big)(\bx-\by)\right){\rm d}\by.\]
Since \eqref{sosinex1} is the solution of \eqref{example1*}, we get
\[D_{0+}^{\beta_0}h_1(t)=\,^{C}\partial_t^{\beta_0}w(\bx,t)|_{_{_{\bx=\bq}}}=t^{\beta_0}\Delta_{\bx} w(\bx,t)|_{_{_{\bx=\bq}}}.\]
On the other hand, we get
\begin{align*}
-h_2&(t)=-\Delta_{\bx} w_1(\bx)|_{_{_{\bx=\bq}}}t\\
&+\int_{\Omega}\Delta_{\by} w_1(\by)\mathcal{F}^{-1}_{\bs}\big(|\bs|^{2}I_{0+}^{\beta_0}\left(t^{\beta_0+1}E_{1,2\beta_0,\beta_0+1}^{\beta_0}(-|\bs|^{2}t^{2\beta_0})\big)(\bx-\by)|_{_{_{\bx=\bq}}}\right){\rm d}\by. 
\end{align*}
Now, applying Green's second formula and \eqref{sosinex1} we get 
\begin{align*}
-h_2(t)&=-\Delta_{\bx} w_1(\bx)|_{_{_{\bx=\bq}}}t\\
&\hspace{-0.5cm}+\int_{\Omega}\Delta_{\by} w_1(\by)\mathcal{F}^{-1}_{\bs}\big(|\bs|^{2}I_{0+}^{\beta_0}\left(t^{\beta_0+1}E_{1,2\beta_0,\beta_0+1}^{\beta_0}(-|\bs|^{2}t^{2\beta_0})\big)(\bx-\by)|_{_{_{\bx=\bq}}}\right){\rm d}\by \\
&=-\Delta_{\bx} w_1(\bx)|_{_{_{\bx=\bq}}}t\\
&\hspace{-0.5cm}+\int_{\Omega}w_1(\by)\Delta_{\by}\mathcal{F}^{-1}_{\bs}\big(|\bs|^{2}I_{0+}^{\beta_0}\left(t^{\beta_0+1}E_{1,2\beta_0,\beta_0+1}^{\beta_0}(-|\bs|^{2}t^{2\beta_0})\big)(\bx-\by)|_{_{_{\bx=\bq}}}\right){\rm d}\by \\
&=-\Delta_{\bx}w(\bx,t)|_{_{_{\bx=\bq}}}.\end{align*}
Hence 
\[\rho(t)=t^{\beta_0}=\frac{D_{0+}^{\beta_0}h_1(t)}{h_2(t)}=\frac{t^{\beta_0}\Delta_{\bx}w(\bx,t)|_{_{_{\bx=\bq}}}}{\Delta_{\bx}w(\bx,t)|_{_{_{\bx=\bq}}}}.\]




\section*{acknowledgements}
The first and second authors were supported by the FWO Odysseus 1 grant G.0H94.18N: Analysis and Partial Differential Equations and by the Methusalem programme of the Ghent University Special Research Fund (BOF) (Grant number 01M01021). The second author is also supported by EPSRC grant EP/R003025/2 and FWO Senior Research Grant G011522N. The third author is supported by the Committee of Science of the Ministry of Science and Higher Education of the Republic of Kazakhstan (Grant No. BR21882172). The third author is also supported by Nazarbayev University under grants 20122022CRP1601 and 20122022FD4105.

 \section*{\small
 Conflict of interest} 

 {\small
 The authors declare that they have no conflict of interest.}




\end{document}